\documentclass[a4paper,11pt,oneside]{amsart}

\usepackage{amsmath, amsthm, amscd,  amsfonts,}
\usepackage{graphicx,float}
\usepackage{lineno,hyperref}
\usepackage{mathrsfs}
\usepackage{amssymb}
\usepackage{color}
\usepackage{dsfont}
\usepackage{fancybox}
\usepackage{comment}

\usepackage[english]{babel}

\newcommand{\even}{\mathrm{Even}}
\newcommand{\TP}{\mathrm{TP}}

\newcommand{\Ult}{\mathrm{Ult}}

\newcommand*\axiomfont[1]{\textsf{\textup{#1}}}
\newcommand\zfc{\axiomfont{ZFC}}
\newcommand\gch{\axiomfont{GCH}}
\newcommand\sch{\axiomfont{SCH}}
\newcommand\ch{\axiomfont{CH}}

\newcommand{\cof}{\mathrm{cof}}

\newcommand{\dom}{\text{dom}\,}
\newcommand{\Add}{\mathrm{Add}}
\newcommand{\Con}{\mathrm{Con}}

\newtheorem{theo}{Theorem}[section]
\newtheorem{quest}[theo]{Question}

\newtheorem{defi}[theo]{Definition}
\newtheorem{lemma}[theo]{Lemma}
\newtheorem{prop}[theo]{Proposition}
\newtheorem{remark}[theo]{Remark}
\newtheorem{convention}[theo]{Convention}
\newtheorem{claim}[theo]{Claim}

\newtheorem{notation}[theo]{Notation}

\def\s{\subseteq}
\newcommand{\one}{\mathds{1}}

\title[The tree property at double successors of singular cardinals]{The tree property at  double successors of singular cardinals of uncountable cofinality with infinite gaps}

\begin{document}

\author{Mohammad Golshani}
\address{School of Mathematics, Institute for Research in Fundamental Sciences (IPM), P.O. Box:
19395-5746, Tehran-Iran.}
\email{golshani.m@gmail.com}
\author{Alejandro Poveda}
\address{Departament de Matem\`atiques i Inform\`atica, Universitat de Barcelona.
Gran Via de les Corts Catalanes, 585,
08007 Barcelona, Catalonia.}
\email{alejandro.poveda@ub.edu}
\subjclass[2000]{Primary: 03E35, 03E55. Secondary: 03E05.}
\keywords{Tree property, Large Cardinals, Singular Continuum Hypothesis, Magidor forcing,  Mitchell forcing.}
\thanks{ The first author's research has been supported by a grant from IPM (No. 97030417). The second author's research has been supported by MECD (Spanish Government) Grant no FPU15/00026, MEC project number MTM2017-86777-P and SGR (Catalan Govern\-ment) project number 2017SGR-270.}

\maketitle

\begin{abstract}
Assume that $\kappa$ and $\lambda$ are respectively strong and weakly compact cardinals with $\lambda>\kappa$. Fix $\Theta\geq \lambda$ a cardinal with $\cof(\Theta)>\kappa$ and $\cof(\delta)=\delta<\kappa$. Assuming the $\gch_{\geq \kappa}$ holds, we construct a generic extension of the universe where $\kappa$ is a strong limit cardinal, $\cof(\kappa)=\delta$, $2^\kappa= \Theta$ and $\TP(\kappa^{++})$ holds.  This extends the main result of \cite{FriHon} for uncountable cofinalities.
\end{abstract}

\section{Introduction}
Infinite trees play an essential role in Combinatorial Set Theory. Recall that for an infinite cardinal $\kappa$, a tree $\langle T, \prec\rangle$ is said to be a $\kappa$-tree if it has height $\kappa$ and all of its levels have size less than $\kappa$. A cofinal branch in $T$ is just a chain in $\prec$ of size $\kappa$. In the present paper we are interested in the following  classical question about infinite trees: Given any cardinal $\aleph_0\leq \cof(\kappa)=\kappa$, does any $\kappa$-tree $T$ have a cofinal branch?  If for a given such $\kappa$ the answer to this question is positive it is said that the \textit{Tree Property holds at $\kappa$}. On what follows will  be denoting this latter fact by $\TP(\kappa)$.

Classical results respectively due to K\"{o}nig and Aronszajn  show that $\mathrm{TP}(\aleph_0)$ holds while, surprisingly,  $\mathrm{TP}(\aleph_1)$ fails. In this regard it is worth saying that both assertions are derivable just from \zfc,\, thus without  appea\-ling to further axiomatic assumptions. At the light of these discovering it seems natural to pursue a similar investigation for bigger regular cardinals. The first theorem on this direction was proved by Mitchel  and Silver \cite{Mit}. This result stands out by its emphasis on the fundamental role of Large Cardinals in the understanding of the tree property configurations above $\aleph_1$. Using the Forcing technique,  Mitchell first showed that $\Con(\zfc+\exists \kappa\,\text{$``\kappa$ is weakly compact''})$ implies
$\Con(\zfc+\TP(\aleph_2))$. The converse implication was subsequently proved by Silver, who showed that if $\TP(\aleph_2)$ holds then $\omega_2$ is  weakly compact in $L$. As a result, the exact consistency strength of $\zfc+\TP(\aleph_2)$ was established, being  the existence of a weakly compact cardinal. %In particular it follows that one can not avoid $\aleph_2$-Aboth theories are equiconsistent, hence one can not obtain $\TP(\aleph_2)$ without previously transcend \zfc.

In this paper we are particularly interested in the forcing notion introduced by Mitchell in his proof of this first implication. In honour to his ground-breaking discovering, this forcing  is known among the specialists as \emph{Mitchell Forcing}. Given a weakly compact cardinal $\kappa$, forcing with Mitchell's poset produces a generic extension where $\kappa=\aleph_2$, $2^{\aleph_0}=\aleph_2$ and $\TP(\aleph_2)$ holds \cite{Mit}. In this regard, it is worth mention that the presence of this failure of the \ch\, in Mitchell's model is not casual: actually, it is mandatory by virtue of a classical result of Specker \cite{Spe}.

 Following up with Mitchell's work, Abraham \cite{Abr} later proved the follo\-wing: Assume the \gch\, holds,  and that there is a supercompact cardinal $\kappa$ joint with a weakly compact cardinal $\lambda>\kappa$. Then there is a forcing notion which produces a generic extension where $\kappa=\aleph_2$, $\lambda=\aleph_3$, and both $\TP(\aleph_2)$  and $\TP(\aleph_3)$ hold.\footnote{Also, in this model $2^{\aleph_0}=\aleph_2$, $2^{\aleph_1}=\aleph_3$.}  This result was subsequently extended by Cummings and Foreman in \cite{CumFor}: Assume the \gch\, holds and that there is a sequence $\langle \kappa_n\mid n<\omega\rangle$ of supercompact cardinals. Then there is a generic extension where  $\TP(\aleph_n)$ holds, for each $2\leq n<\omega$. Subsequent improvements of this result are due to Neeman \cite{NeeUpTo} and Unger \cite{UngUpTo}.

 The \textit{raison d'être} for   all of these investigations raises from an old open question posed by Magidor: Is it consistent that $\TP(\kappa)$ holds, for each $\aleph_2\leq \cof(\kappa)=\kappa$? Alternatively, does \zfc\, prove the existence of $\kappa$-Aronszajn trees, when $\aleph_2\leq \cof(\kappa)=\kappa$? At the light of the previous results it seems that there is a chance --of course, modulo Large Cardinlas-- for a positive answer to Magidor's question. Nonetheless, the construction of a model witnessing this thesis remains as one of the main open challenges of Set Theory.
 %The construction of a model for a positive \linebreak answer to Magidor's question is one of the main open challenges of Set Theory.
A remarkable aspect of this problem come from its narrow connection  with the possible behaviours of the continuum function $\aleph_\alpha\mapsto 2^{\aleph_\alpha}$, which constitutes another central theme in Cardinal Combinatorics.  More precisely, if  $M$ is a model for a positive answer to Magidor's question
then $M$ witnesses a global failure of the \gch. The first model $M$ exhibiting this global failure of the \gch\, was constructed in \cite{ForWoo} using a supercompact cardinal and Supercompact Radin forcing with collapses.

  %The numerous attempts towards settling this question has not been in vain. Actually, des\-pite a complete solution seems to be far from our current knowledge of the theory, many sophisticated ideas and techniques has been discovered along the process.

  %of it has motivated the discovering of new sophisticated ideas and techniques. Nonetheless the construction of a model for the Tree Property at all regular cardinals ${\geq}\aleph_2$, seems, up today, far from our current knowledge of the theory.

The aim of this paper is to contribute to this
collective effort by analy\-zing the tree property configurations at double successors of singular strong limit cardinals. The first of these results is due to Cummings and Foreman \cite{CumFor}, which has inspired the latter developments of the field. In the above cited paper the authors proved the following analogous of Mitchell's theorem at the scale of strong limit singular cardinals: Assume the $\gch$ holds and that there is a supercompact cardinal $\kappa$ joint with a weakly compact $\lambda>\kappa$. Then there is a \textit{Mitchell-like} forcing which yields a generic extension where $\kappa$ is strong limit, $\cof(\kappa)=\omega$, $2^\kappa=\kappa^{++}=\lambda$, hence $\sch_\kappa$ fails, and $\mathrm{TP}(\kappa^{++})$ holds. Later developments due to Friedman and Halilovi\'{c} \cite{FriHal} obtained Cummings-Foreman's result for $\kappa=\aleph_\omega$, starting from almost optimal hypotheses. Subsequently, Gitik \cite{Gitik2014} provided the exact consistency strength of this theory.

 The main novelty of Cummings-Foreman (CF) approach is that it provides a general scheme to combine  the Prikry-type technology  with Mitchell's original ideas. This aspect has been subsequently exploited in \cite{Ung}, \cite{SinTree}, \cite{GolMoh} and \cite{FriHon} where several generalizations of CF-theorem have been obtained.

In this paper we are particularly interested in the approach taken in \cite{FriHon}, where the problematic of getting arbitrary failures for the $\sch_\kappa$ in the CF-model is addressed. As it is mentioned in the introduction of \cite{FriHon} this is somewhat conflictive with Mitchell's original approach. Broadly speaking, if one aims to force a generic extension where  $2^\kappa\geq \kappa^{+3}$ and $\TP(\kappa^{++})$ hold, then the Mitchell's-like forcing from \cite{CumFor} would exhibit a mismatch between the lengths of its \textit{Cohen component} and its \textit{Collapsing component}. It turns that this disagreement  become troublesome at the time of implementing Mitchell's classical analysis of the quotients forcings (see \cite{Mit} or \cite{Abr} for details). The above cited paper is precisely devoted to show how to overcome this difficulty.

In the present work we will provide a further generalization of the CF-theorem which extends the developments obtained in \cite{FriHon} and \cite{GolMoh}.
The main result of the paper reads as follows:

%A natural question arises from Cummings-Foreman theorem: For a strong limit singular cardinal $\kappa$, is $\TP(\kappa^{++})$ consistent with arbitrary failures of $\mathrm{SCH}_\kappa$? In other terms, is it possible to get a bigger gap between $\kappa$ and $2^\kappa$ in the model described so far? This question was answered by Friedman-Honzik-Stejskalov\'a  in \cite{FriHon} where a method to get arbitrary gaps in Cummings-Foreman model is described. In this paper we aim to proved that the analogous situation is consistent for uncountable cofinalities:
\begin{theo}[Main theorem]
\label{main theorem 1}
Assume the $\gch_{\geq \kappa}$ holds. Let $\kappa$ be a strong cardinal and $\lambda > \kappa$ be a weakly compact. Fix $\Theta \geq  \lambda$ a cardinal with $\cof(\Theta)>\kappa$ and  $\cof(\delta)=\delta < \kappa$. Then there is a generic extension of the universe of sets $V$ where the following properties hold:
\begin{enumerate}
\item $\kappa$ is a strong limit singular cardinal with $\cof(\kappa)=\delta$;
\item All cardinals and cofinalities outside $((\kappa^+)^V,\lambda)$ are preserved. In particular, $\lambda=(\kappa^{++})^V$.
\item $2^{\kappa} = \Theta$, hence $\sch_{\kappa}$ fails;
\item $\mathrm{TP}(\kappa^{++})$ holds.
\end{enumerate}
\end{theo}
  For the proof of this result we have been inspired by \cite{FriHon}, where in our context the role of  Prikry forcing is now played by Magidor forcing. It is worth mention that in our model $\TP(\kappa^+)$ fails (see Section \ref{NotTPkappa^+}).  Roughly speaking this is consequence of the fact that our forcing does not collapse $\kappa^+$. To overcome this issue one needs to use \emph{Supercompact-like Prikry forcings} such as Sinapova forcing \cite{Sin}. A generalization of  our main theorem in this direction has been recently obtained in \cite{Pov}.

The structure of the paper is as follows: In Section \ref{Preliminaries} we give a self-contained account of the classical Magidor forcing to change the cofinality of a measurable cardinal $\kappa$ with $o(\kappa)=\delta$ to $\delta$.  In Section \ref{ProofTheorem1} we  present in detail our main forcing construction to obtain Theorem \ref{main theorem 1} when $\Theta=\lambda^+$. There we will show that this forcing produces a generic extension where (1)-(3) hold. In Section \ref{TPkappa++SectionMagidor} we finish this analysis by showing that $\TP(\kappa^{++})$ also holds.

 We will end up  with Section \ref{SectionGapArbitrary}, where we will enumerate the necessary modifications to obtain the  consistency of an arbitrary failure of the $\sch_\kappa$ in our model. The notation used is either quite standard or will be properly explained. Certain familiarity with Prikry-type forcings \cite{Git} and Large Cardinals \cite{Kan} would be desirable.

\section{Preliminaries}\label{Preliminaries}
One of the main forcing tools of the present paper is the so-called \textit{Magidor forcing}. This poset was originally introduced by Magidor in \cite{Mag} and is inteded to change the cofinality of a measurable cardinal $\kappa$ with $o(\kappa)=\delta$ to some regular cardinal $\delta'\leq \delta$.\footnote{For a definition of $o(\kappa)$ see \cite[\S2]{MitChap}.} Thus, in a broad sense, Magidor forcing can be conceived as a generalization of the  classical Prikry forcing \cite{Prikry}.

 Magidor's original approach to this  forcing was based on $\lhd$-increasing sequences of measures $\mathcal{U}=\langle U_\alpha\mid \alpha<\delta\rangle$ (see \cite[Definition 2.2]{MitChap}). Nonetheless, latter developments of Mitchell suggested to use \emph{Coherent Sequences of Measures}  instead (c.f. Definition \ref{CoherentSequence}). This is due to the fact that coherent sequence of measures can be also used to define more sophisticated forcings, such as Radin forcing \cite{Rad}. In this paper we will follow Mitchell's approach to Magidor forcing.

The purpose of this section is just to review the definition  and main properties of Magidor forcing. This will become important in sections \ref{ProofTheorem1} and \ref{TPkappa++SectionMagidor} where we will modify Cummings-Foreman poset $\mathbb{R}$ by switching Prikry by Magidor forcing. %For the reader's benefit we will provide some details and intuitions behind this forcing.
All details can be found  in Magidor's original paper \cite{Mag} or in Gitik's excellent article \cite[\S5]{Git}, with the exception of  Lemma \ref{RowbottomMagidor1} and  Lemma \ref{RowbottomMagidor}.

\begin{defi}[Coherent sequence]\label{CoherentSequence}
A coherent sequence of measures $\mathcal{U}$ is a function with domain $\{(\alpha,\beta)\mid\, \alpha<\ell^{\mathcal{U}}~\mathrm{and}~\beta<o^{\mathcal{U}}(\alpha)\}$ such that for $(\alpha,\beta)\in dom(\mathcal U)$ the following conditions are true:
\begin{enumerate}
\item
$\mathcal{U}(\alpha,\beta)$ is a normal measure over $\alpha$;
\item
If $j^{\alpha}_{\beta}:V\longrightarrow \Ult(V,\mathcal{U}(\alpha,\beta))$ stands for the usual ultrapower embedding, then $j^{\alpha}_{\beta}(\mathcal{U})\upharpoonright\alpha+1=\mathcal{U}\upharpoonright(\alpha,\beta)$,
where $\mathcal{U}\upharpoonright\alpha:= \mathcal{U}\upharpoonright\{(\alpha',\beta')\mid \alpha'<\alpha\,\&\, \beta'<o^{\mathcal{U}}(\alpha')\}$ and
$$
\mathcal{U}\upharpoonright(\alpha,\beta):=\mathcal{U}\upharpoonright\{(\alpha',\beta')\mid(\alpha'<\alpha\,\&\,\beta'<o^{\mathcal{U}}(\alpha'))~\mathrm{or}~(\alpha=\alpha'\,\&\,\beta'<\beta)\}.
$$
\end{enumerate}
The ordinals $\ell^{\mathcal U}$ and $o^{\mathcal{U}}(\alpha)$ are called respectively  the length of $\mathcal{U}$ and the Mitchell order at $\alpha$ of $\mathcal U$.
\end{defi}
\begin{defi}
Let $\mathcal{U}=\langle \mathcal{U}(\alpha,\beta)\mid \alpha\leq \kappa, \beta<o^{\mathcal{U}}(\alpha)\rangle$ be a coherent sequence of measures with $\ell^\mathcal{U}=\kappa+1$ and $o^{\mathcal{U}}(\kappa)=\delta$. For each $\alpha<\ell^\mathcal{U}$, define $\mathcal{F}_\mathcal{U}(\alpha):=\bigcap\limits_{\beta<o^{\mathcal{U}}(\alpha)}\mathcal{U}(\alpha,\beta)$, if $o^{\mathcal{U}(\alpha)}>0$, and otherwise,  set $\mathcal{F}_\mathcal{U}(\alpha):=\{\emptyset\}$.
\end{defi}
Observe that $\mathcal{F}_\mathcal{U}(\alpha)$ is the set of all subsets of $\alpha$ which are measure one with respect to all the measures $\mathcal{U}(\alpha,\beta)$, $\beta<o^{\mathcal{U}}(\alpha)$. It is fairly easy to check that in the non-trivial case where $o^{\mathcal{U}}(\alpha)> 0$,  $\mathcal{F}_\mathcal{U}(\alpha)$ yields an $\alpha$-complete normal filter over $\alpha.$ In the cases where $\mathcal{U}$ is clear from the context we will tend to omit its mention when referring to $\mathcal{F}_\mathcal{U}(\alpha)$.
\begin{defi}[Magidor forcing]\label{MagidorForcing}
Let $\mathcal{U}$ be a coherent sequence of measures with $\ell^{\mathcal{U}}=\kappa+1$ and $o(\kappa)=o^{\mathcal{U}}(\kappa)=\delta$. %Further, assume that $\delta$ is a limit ordinal.
\begin{enumerate}
\item [(a)] Magidor forcing relative to $\mathcal{U}$, denoted by $\mathbb{M}_{\mathcal{U}},$ consists of all finite sequences of the form $p=\langle \langle\alpha^p_0,A^p_0\rangle,\dots,\langle\alpha^p_n,A^p_n\rangle\rangle$ where:
\begin{enumerate}
\item[($\aleph$)] $\delta < \alpha_0<\dots<\alpha_n=\kappa$,
\item[$(\beth)$] $A^p_i\in\mathcal{F}(\alpha^p_i)$,
\item[$(\gimel)$] $A^p_i\cap (\alpha^p_{i-1}+1)=\emptyset$ (where,  $\alpha^p_{-1}:=\delta+1$).
\end{enumerate}
The sequence $\langle \alpha^p_0,\dots,\alpha^p_{n-1}\rangle$ is called the \emph{stem} of $p$ and the integer $n^p$ the \emph{length} of $p$. Whenever the condition is clear from the context we shall tend to suppress the co\-rresponding superscript.
\item [(b)] For $p=\langle \langle\alpha_0,A_0\rangle,\dots,\langle\alpha_n,A_n\rangle\rangle$ and $q=\langle \langle\beta_0,B_0\rangle,\dots,\langle\beta_m,B_m\rangle\rangle$ be two conditions in $\mathbb{M}_{\mathcal{U}}$ we will say that $q$ is stronger than $p$ ($q\leq p$) if the following conditions are fulfilled:
\begin{enumerate}
\item[$(\aleph)$] $m\geq n$,
\item[$(\beth)$] $\forall i\leq n~~\exists j\leq m~~\alpha_i=\beta_j$ and $B_j\subseteq A_i$,
\item[$(\gimel)$] For all $j$ be such that $ \beta_j\notin \{\alpha_1,\dots,\alpha_n\},~~B_j\subseteq A_k\cap \beta_j$ and $ \beta_j\in A_k$, where $k:=\min\{k\leq n\mid \beta_j<\alpha_k\}$.
\end{enumerate}
\item [(c)]  $q$ is a direct extension or a Prikry extension of $p~(q\leq^* p)$  if $q\leq p$ and $m=n$.
\end{enumerate}
In the cases where $\mathcal{U}$ is clear from the context, we will tend to write $\mathbb{M}$ rather than $\mathbb{M}_\mathcal{U}$. Given a condition $p\in\mathbb{M}$ we will write $$p:=\langle\langle\alpha^p_0,A^p_0 \rangle, \dots, \langle \alpha^p_{n^p-1},A^p_{n^p-1} \rangle,\langle \alpha^p_{n^p},A^p_{n^p} \rangle  \rangle.$$
If $p, q\in\mathbb{M}_\mathcal{U}$ are two conditions with the same stem, we define $p\wedge q:= \langle\langle\alpha^p_0,A^p_0\cap A^q_0 \rangle, \dots, \langle \alpha^p_{n^p},A^p_{n^p}\cap A^q_{n^q} \rangle  \rangle$.
\end{defi}
\begin{defi}\label{OneStepExtensionMagidor}
Let $p$ be a sequence witnessing clauses $(\aleph)$ and $(\gimel)$ of Defi\-nition \ref{MagidorForcing}(a).  For $i\leq n^p$ and $\alpha\in A^p_i$, define $p{}^\curvearrowright\langle\alpha\rangle$ as the sequence
$\langle\langle\alpha^p_0,A^p_0\rangle\dots
\langle \alpha^p_{i-1}, A^p_{i-1}\rangle,\langle\alpha, A^p_i\cap \alpha\rangle, \langle \alpha^p_i,A^p_i\rangle,\dots\langle \kappa,A^p_{n^p}\rangle \rangle.$ For a sequence $\vec{\alpha}\in [A^p_i]^{<\omega}$, define by recursion  $p{}^\curvearrowright\vec{\alpha}:=(p{}^\curvearrowright(\vec{\alpha}\upharpoonright|\vec{\alpha}|){}^\curvearrowright \langle\vec{\alpha}(|\vec{\alpha}|)\rangle.\footnote{Here by convention $p{}^\curvearrowright\emptyset:=p$.}$
\end{defi}
\begin{remark}\label{NotAllExtensions}
\rm{
Observe that not for all $\alpha\in  A^p_i$,  $p{}^\curvearrowright\langle\alpha\rangle\in \mathbb{M}$ as it may be the case that $\alpha\cap A^p_i\notin\mathcal{F}(\alpha)$. Actually, $p{}^\curvearrowright\langle\alpha\rangle\in \mathbb{M}$ if and only if $ A^p_i\cap \alpha\in\mathcal{F}(\alpha)$. }
\end{remark}

\begin{defi}
Let $p\in\mathbb{M}$ be a  condition. We will say that a finite sequence of ordinals $\vec{\alpha}$ is \emph{addable} to  $p$ if $\vec{\alpha}\in[\prod_{i\leq n^p} A^p_i]^{\leq n^p}$. We will denote the set of addable sequences to $p$ by $\mathfrak{A}_p$.
\end{defi}
\begin{defi}
Let $p\in\mathbb{M}$. A finite sequence of ordinals  $\vec{x}$ is a \emph{block sequence} for $p$ if $\vec{x}\in [\biguplus_{i\leq n^p} A^p_i]^{<\omega}$. For each $i\leq n^p$, set $\vec{x}_i:=\vec{x}\cap A^p_i$. Also, let $i_{\vec{x}}$ be denote an enumeration of the $i\leq n^p$ for which $\vec{x}_i\neq \emptyset$.
\end{defi}
We use the symbol $\biguplus$ rather than $\bigcup$ just to emphasize the fact that the sets $A^p_i$ and $A^p_j$ are disjoint, provided $i\neq j$.
\begin{defi}[Minimal extensions of $\mathbb{M}$]
Let $p\in\mathbb{M}$ and $\vec{x}$ be a block sequence for $p$. Mimicking Definition \ref{OneStepExtensionMagidor}, we define recursively $p{}^\curvearrowright\vec{x}:=(p{}^\curvearrowright \vec{x}\setminus \cup_{j<i}\vec{x}_j){}^\curvearrowright \vec{x}_i$, where $i=\max i_{\vec{x}}$.
\end{defi}
This discussion lead us to isolate the following concept:
\begin{defi}[Pruned condition]
A condition $p\in\mathbb{M}$ is said to be pruned if for every  $\vec{x}\in[\biguplus_{i\leq n^p} A^p_i]^{<\omega}$, $p{}^\curvearrowright\vec{x}\in \mathbb{M}$.
\end{defi}
\begin{prop}\label{PrunedIsOnePruned}
Let $p\in\mathbb{M}$. Then, $p$ is a  pruned condition if and only if $p{}^\curvearrowright\langle \alpha\rangle\in\mathbb{M}$, for all $\alpha\in \biguplus_{i\leq n^p}A^p_i$.
\end{prop}
\begin{proof}
The first implication is obvious and the second  follows easily from the recursive definition of $p{}^\curvearrowright\vec{x}$. %For the second let us argue by recursion on $m\leq n^p$ that for every $\vec{\alpha}\in\mathfrak{A}_p$ with cardinality $m$, $p{}^\curvearrowright\vec{\alpha}\in\mathbb{M}$. The case $m=1$ follows from our hypothesis. For the inductive step, assume that the result already holds for sequences of length $m$ and combine our hypothesis with the recursive definition of $p{}^\curvearrowright\vec{\alpha}$.
\end{proof}
One can use the previous result to show that any condition $p\in\mathbb{M}$ has a $\leq^*$-extension which is pruned.
\begin{prop}\label{ThereisPrunedInMagidor}
For each $p\in\mathbb{M}$, there is $p^*\leq^* p$ which is pruned.
\end{prop}
\begin{proof}
Fix $p\in\mathbb{M}$. Without loss of generality assume that all large sets in $p$ are non empty, as otherwise the argument is similar. For each $i\leq n^p$, set $A^{*,0}_i:=A^p_i$, $A^{*,n+1}_i:=\{\alpha\in A^{*,n}_i\mid A^{*,n}_i\cap \alpha\in\mathcal{F}(\alpha)\}$
and $A^*_i:=\bigcap_{n<\omega} A^{*,n}_i$.
\begin{claim}
For each $i\leq n^p$ and $n<\omega$, $A^{*,n}_i\in \mathcal{F}(\alpha^p_i)$. In particular, for each $i\leq n^p$, $A^*_i\in\mathcal{F}(\alpha^p_i)$.
\end{claim}
\begin{proof}[Proof of claim]
If the first assertion is true the second  follows automatically by the $\alpha^p_i$-completeness of $\mathcal{F}(\alpha^p_i)$. %We may assume without loss of generality that $A^{*,n}_i\neq \emptyset$, for otherwise the result would be trivial.
 Thus, we want to argue by induction that $A^{*,n}_i\in\mathcal{F}(\alpha^p_i)$, for $n<\omega$. Clearly this is true for $n=0$. For the inductive step, let $\beta<o^{\mathcal{U}}(\alpha^p_i)$ and observe that $\alpha^p_i\in j^{\alpha^p_i}_\beta(A^{*,n}_i)$ and $ j^{\alpha^p_i}_\beta(A_i^{*,n})\cap \alpha^p_i= A_i^{*,n}\in \mathcal{F}(\alpha^p_i)$, so $A^{*,n+1}_i\in\mathcal{F}(\alpha^p_i)$.
\end{proof}
Let $p^*$ be the $\leq^*$-extension of $p$ with $A^{p^*}_i=A^*_i$, each $i\leq n^p$. We claim that $p^*$ is pruned. For showing this we use Proposition \ref{PrunedIsOnePruned}. Let $\alpha\in \biguplus_{i\leq n^p}A^*_i$ and say that $\alpha\in A^*_i$. By construction it is easy to check that $A^{*,n}_i\cap \alpha\in\mathcal{F}(\alpha)$, for all $n<\omega$, hence $A^{*}_i\cap \alpha\in\mathcal{F}(\alpha)$. Finally,  Remark \ref{NotAllExtensions}  yields $p{}^\curvearrowright\langle\alpha\rangle\in\mathbb{M}$.
\end{proof}
%We take advantage of the previous result at Section \ref{TPkappa++SectionMagidor} where we will study the quotient forcings $\mathbb{R}/\mathbb{R}\upharpoonright\xi$.
Let us now state the basic properties of Magidor forcing.
%address the question of cardinals preservation in Magidor extensions. %The first word in this regard is the following proposition:
\begin{prop}\label{MagidorIsKnaster}
$\mathbb{M}$ is $\kappa^+$-Knaster. That is, any set $S\in[ \mathbb{M}]^{\kappa^+}$ contains a set $\mathcal{I}\in[S]^{\kappa^+}$ of pairwise compatible conditions.
\end{prop}
%\begin{proof}
%Let $\{p_\alpha\mid \alpha<\kappa^+\}$ be an enumeration of $S$. For each $\alpha<\kappa^+$, let $s_\alpha$ be the stem of $p_\alpha$ and define $\varphi: \kappa^+\rightarrow [\kappa]^{<\omega}$ as $\varphi(\alpha):=s_\alpha$. By counting arguments, there is $\mathcal{I}\in[S]^{\kappa^+}$ and $s^*\in[\kappa]^{<\omega}$ for which $\varphi[\mathcal{I}]=\{s^*\}$. Let $\alpha,\beta\in\mathcal{I}$ and observe that $p_\alpha\wedge p_\beta\leq_\mathbb{M} p_\alpha, p_\beta$, as wanted.
%\end{proof}
In particular the above implies that all cardinals ${\geq}\kappa^+$ are preserved after forcing with $\mathbb{M}$. Let us now describe the combinatorics of the generic extensions by $\mathbb{M}$ below $\kappa^+$. Hereafter we will assume that $\delta$ is an infinite cardinal and that  $G\s \mathbb{M}$ is a generic filter over $V$. Set $C_G:=\{\alpha<\kappa\mid \exists p\in G\,\exists n<n^p\, (\alpha=\alpha^p_n) \}.$
One may argue as in \cite[Lemma 5.10]{Git} that, below a direct extension of $\one_{\mathbb{M}}$, $C_G$ is a closed unbounded subset of $\kappa$ of order type $\omega^\delta$, which will be referred as the Magidor club induced by $G$.\footnote{Here $\omega^\delta$ stands for ordinal exponentiation rather than cardinal exponentiation.} In particular, there is a direct extension of $\one_\mathbb{M}$ which forces $\cof(\check{\kappa})=\cof^V(\check{\delta})$. %Moreover, if $\cof^V(\delta)=\delta$, this   entails $\one_\mathbb{M}\Vdash_{\mathbb{M}}\cof(\check{\kappa})=\delta$.
However, notice that it is still necessary to prove that $\cof^V(\delta)$ and $\kappa$ have not been collapsed after adding this generic club. %To this aim, one may want to analyse whether cardinals below $\kappa$ are preserved.
As usual, the key property that provides us of the necessary control on the combinatorics of $V_\kappa^{\mathbb{M}}$ is the  Prikry property.
\begin{prop}\label{PrikrypropertyforMagidor}
$\langle \mathbb{M}, \leq^*\rangle$ satisfies the Prikry property: namely, for each sentence $\varphi$ in the language of forcing $\mathcal{L}_{\Vdash}$ and a condition $p\in \mathbb{M}$, there is $q\leq^* p$ such that $q\parallel \varphi$.
\end{prop}
For a proof of the above result see \cite{Mag} or \cite[Lemma 5.8]{Git}.  There is another important structural property of  Magidor forcing which we would like to mention. This property is the following: for a condition $p\in \mathbb{M}$, $\mathbb{M}\downarrow p$ is isomorphic to the product $\mathbb{M}_1\times\mathbb{M}_2$, where $\mathbb{M}_1$ and $\mathbb{M}_2$ are two Magidor forcing. This \emph{fractal structure} is actually shared with other classical Prikry-type forcing that center around uncountable cofinalities, such as Radin forcing \cite{Rad} or Sinapova forcing \cite{Sin}. Here $\mathbb{M}_1$ is a Magidor forcing adding a club subset of  $\theta$, where $\theta$ is a point in the Magidor club $C_{\dot{G}}$ and $\dot{G}$ is a $\mathbb{M}\downarrow p$-name for a generic filter. On the other hand, $\mathbb{M}_2$ is essentially the Magidor forcing $\mathbb{M}$, though this time adding a generic club on $\kappa$ of points above $\theta$. Let us phrase this discussion in more formal terms.
\begin{defi}
Let $p\in\mathbb{M}$ and $m\leq n^p$. We will respectively denote by $p^{\leq m}$ and $p^{>m}$ the sequences
\begin{eqnarray*}
p^{\leq m}:=\langle \langle \alpha^p_0,A^p_0\rangle,\dots, \langle \alpha^p_{m}, A^p_m\rangle \rangle,\\
p^{>m}:=\langle\langle \alpha^p_{m+1}, A^p_{m+1}\rangle, \dots, \langle \alpha^p_{n^p}, A^p_{n^p}\rangle \rangle.
\end{eqnarray*}
\end{defi}
If $m<n^p$, is immediate that $p^{\leq m}\in\mathbb{M}_{\mathcal{U}\upharpoonright\alpha^p_m+1}$ and $p^{>m}\in\mathbb{M}_\mathcal{U}$. %The next easy lemma witnesses the said fractal structure of Magidor forcing:
\begin{lemma}\label{SplittingMagidor}
Let $p\in\mathbb{M}$ and $m<n^p$. There is an isomorphism  between $\mathbb{M}_\mathcal{U}\downarrow p$ and $\mathbb{M}_{\mathcal{U}\upharpoonright\alpha^p_m+1}\downarrow p^{\leq m}\times\mathbb{M}_\mathcal{U}\downarrow p^{>m}$. In particular, $\mathbb{M}_\mathcal{U}\downarrow p$ projects onto $\mathbb{M}_{\mathcal{U}\upharpoonright\alpha^p_m+1}\downarrow p^{\leq m}$.
\end{lemma}
%\begin{proof}
%For a condition $q\leq p$, let $q^{-}$ be the initial segment of  $q$ for which $\alpha^p_m$ is the last ordinal with $\langle \alpha^p_m, A\rangle\in q^{-}$, for some $A\in\mathcal{F}(\alpha^p_m)$. Analogously, let $q^+$ be the sequence such that $q^{-}{}^\smallfrown q^{+}=q$. It is routine to check that $q\mapsto \langle q^-,q^+\rangle$ yields the desired isomorphism.
%\end{proof}
%Let $p\in\mathbb{M}$ and $\alpha<\kappa$. We will say  that $\alpha$ appears in  $p$ if for some $m<n^p$ and $A\in\mathcal{F}(\alpha)$, $p(m)=\langle\alpha, A\rangle$. Analogously, define the notion ``$\alpha$ appears in $p$ at $m$''. In a mild abuse of notation, let us write $\langle \mathbb{M}_{\mathcal{U}}\downarrow p,\leq^*\rangle$ for the subforcing of $\mathbb{M}_{\mathcal{U}}$ consisting of conditions $\leq^*$-below $p$.
\begin{lemma}\label{PropertiesofSplittings}
In the conditions of the above lemma, $\mathbb{M}_{\mathcal{U}\upharpoonright\alpha^p_m+1}$ is ${(\alpha^p_m)}^+$-Knaster and $\langle \mathbb{M}_\mathcal{U}\downarrow p^{>m},\leq^*\rangle$ is $(\alpha^p_{m})^+$-closed. Moreover, for each $\varrho<\alpha^p_{m+1}$, there is $q_\varrho\leq^* p$ such that $\langle \mathbb{M}_\mathcal{U}\downarrow q_\varrho^{>m},\leq^*\rangle$ is $|\varrho\times \mathbb{M}_{\mathcal{U}\upharpoonright\alpha^p_m+1}|^+$-closed.
\end{lemma}
%\begin{proof}
%The first claim is consequence of Proposition \ref{MagidorIsKnaster}. For the second,  observe that $\beta:=\beta^{p^{>m}}$ always exists, as $\kappa$ is always part of the defining set. Let $\gamma<\beta$ and $\langle
%p_\alpha\mid \alpha<\gamma\rangle$ be a $\leq^*$-decreasing sequence of conditions $\leq^*$-below $p^{>m}$.  Observe that there is a sequence $s$ such that $s(i)=\langle \gamma_i,\emptyset\rangle$, $i<|s|$, and $p_\alpha=s{}^\frown q_\alpha$, for each  $\alpha<\gamma$. Of course, $|s|$ may be $0$. In any case, notice that $q_\alpha$ is a sequence of pairs where all the filters involved are $\beta$-complete. Thus, $p^*:=s{}^\frown \bigwedge_{\alpha<\gamma} q_\alpha$ provides the desired lower bound. The last claim follows from noticing that $\beta$ is a measurable ${\geq}\alpha^p_{m+1}$.
%\end{proof}
By combining Proposition \ref{PrikrypropertyforMagidor} and lemmas \ref{SplittingMagidor} and \ref{PropertiesofSplittings} one may easily obtain a complete picture of  the combinatorics of $V[G]_\kappa$. For convenience, let $\langle\kappa_\alpha\mid \alpha<\theta\rangle$ be an enumeration of the generic club induced by $G$. We will say that $\alpha<\kappa$ appears in $p\in\mathbb{M}$ if for some $m<n^p$ and $A\in\mathcal{F}(\alpha)$, $p(m)=\langle\alpha,A\rangle$. We will say that $\alpha<\kappa$ appears in $p$ at $m<n^p$  if $p(m)=\langle\alpha, A\rangle$, for some $A\in\mathcal{F}(\alpha)$.
\begin{prop}\label{CombinatoricsBelowkappa}
For each  $\varrho<\kappa$, set  $\alpha_\varrho:=\min\{\alpha<\theta\mid \kappa_\alpha\leq \varrho<\kappa_{\alpha+1}\}$ and let $p\in G$ where both $\kappa_{\alpha_\rho}$ and $\kappa_{\alpha_\rho+1}$ appear at $m$ and $m+1$, respectively. Then, $$\mathcal{P}(\varrho)^{V[G]}=\mathcal{P}(\varrho)^{V[G_{\alpha_\varrho}]},$$ where $G_{\alpha_\varrho}$ is the filter generated by $G$ and the natural projection between $\mathbb{M}\downarrow p$ and $\mathbb{M}_{\mathcal{U}\upharpoonright \kappa_{\alpha_\varrho}+1}\downarrow p^{\leq m}$.
\end{prop}
The above proposition yields Magidor's theorem:
\begin{theo}[Magidor]\label{MagidorTheoremForcing}
Let $\delta<\kappa$ be a regular and a measurable cardinal with $o(\kappa)=\delta$, respectively. There is a cardinal-preserving generic extension of the universe where $\kappa$ is a strong limit cardinal with $\cof(\kappa)=\delta$.
\end{theo}

There is another remarkable feature of Magidor forcing which is essential for our purposes. This aspect is related with the way Magidor generics are generated. Recall that if $G\s\mathbb{M}$ is generic over $V$ then it generates a Magidor club that we denoted by $C_G$. Conversely, if $C_G\s\kappa$ is the Magidor club generated by $G$, then the set of conditions $p\in G(C_G)$ defined as
\begin{enumerate}
\item[$(\aleph)$] $\forall m<n^p\, \alpha^p_m\in C_G$,
\item[$(\beth)$] $\forall\vartheta\in C_G\,\exists q\leq_{\mathbb{M}} p\,\exists m<n^q\, \text{($\vartheta$ appears in $q$)}$,
\end{enumerate}
generates a filter which contains $G$, hence it is generic and $G(C_G)=G$. In particular,  $V[G]=V[C_G]$.

Let us say that a sequence $\vec{\gamma}$ in $\kappa$ is a Magidor sequence for $\mathbb{M}_\mathcal{U}$ over $V$ if the set $G(\vec{\gamma})$ is a $\mathbb{M}_\mathcal{U}$-generic filter over $V$.\footnote{Here we are identifying $\vec{\gamma}$ with a club subset of $\kappa$. }  By definition, any Magidor sequence for $\mathbb{M}_\mathcal{U}$ over $V$ generates a Magidor generic over $V$. Conversely,  any Magidor generic $G$ for $\mathbb{M}_\mathcal{U}$ generates a Magidor sequence $\vec{\gamma}_G$ over $V$, as witnessed by any increasing enumeration of $C_G$. From this it is clear that any Magidor extension is ultimately determined by a Magidor sequence. It is thus natural to ask whether there is a criterion that allows to establish when a sequence $\vec{\gamma}$ is indeed a Magidor sequence for some $\mathbb{M}_\mathcal{U}$. The following result due to Mitchell  \cite{MitHow} provides the desired characterization: %Even more, at the light of Mathias' criterion of genericity for Prikry forcing  \cite{Mathias} the hope for this new criterion seems more than reasonable. Indeed, the analogous of Mathias theorem in the context of Magidor forcing already exists and was discovered by Mitchell in \cite{MitHow}:
\begin{theo}[Mitchell]\label{MitchellCriterion}
Assume that $V$ is an inner model of $W$ with $\mathbb{M}_\mathcal{U}\in V$. A sequence $\vec{\gamma}\in W$ is a Magidor sequence over $V$ if and only if the following hold true:
\begin{enumerate}
\item for $\alpha<|\vec{\gamma}|$, $\vec{\gamma}\upharpoonright \alpha$ is a Magidor sequence for  $\mathbb{M}_{\mathcal{U}\upharpoonright \gamma(\alpha)+1}$ over $V$;
\item for each $X\in \mathcal{P}(\kappa)^V$, $X\in\mathcal{F}_\mathcal{U}(\kappa)^V$ if and only if for a tail end of $\alpha<|\vec{\gamma}|$, $\vec{\gamma}(\alpha)\in X$.
\end{enumerate}
\end{theo}
We will be using this result in the next section when we show that $\mathbb{R}$ projects onto $\mathrm{RO}^+(\mathbb{R}\upharpoonright\xi)$ (c.f. Proposition \ref{ProjectionRandTruncationsMagidor}).
\medskip

 Finally, we prove a generalization of the classical Röwbottom's Lemma \cite[Theorem 7.17]{Kan} which will be used in our future analysis of the quotients $\mathbb{R}/\mathbb{R}\upharpoonright\xi$. %For simplicity we phrase it in terms of Magidor conditions but the result stills holds for obvious class of functions under consideration.
The following lemma is the key in proving Lemma \ref{RowbottomMagidor}.

\begin{lemma}\label{RowbottomMagidor1}
Let $f: [\kappa]^{<\omega} \to \tau, \tau < \kappa,$ and let $\langle U_\alpha \mid \alpha < \delta \rangle, \delta < \kappa,$ be a sequence of normal measures on $\kappa$. Then there are sets $A_\alpha \in U_\alpha,$ for $\alpha < \delta,$ such that whenever $\langle \alpha_0,\cdots, \alpha_{n-1}  \rangle$
is a finite sequence of ordinals less than $\delta,$ the function $f$ is constant on the set of increasing sequences $\langle \nu_0, \cdots, \nu_{n-1}  \rangle \in \prod_{i\leq n-1}A_{\alpha_i}$.
\end{lemma}
\begin{proof}
We prove, by induction on $n < \omega,$ that for each $f: [\kappa]^{n} \to \tau,$ where $\tau < \kappa$, there are sets $A_\alpha \in U_\alpha,$ for $\alpha < \delta$  and a function $g:[\delta]^n \to \tau$ such that for each finite sequence $\langle \alpha_0,\cdots, \alpha_{n-1}  \rangle$ of ordinals less than $\delta$ and all increasing sequences $\langle \nu_0, \cdots, \nu_{n-1}  \rangle \in A_{\alpha_0} \times \cdots \times A_{\alpha_{n-1}}$, we have
\[
f(\langle \nu_0, \cdots, \nu_{n-1}  \rangle)=g(\langle \alpha_0,\cdots, \alpha_{n-1}  \rangle).
\]
From this the result follows easily by using the $\sigma$-completeness of the ultrafilters $U_\alpha$. Observe that when $n=1$, this is clear: for each $\alpha < \delta$ let $A_\alpha \in U_\alpha$ be such that $f \restriction A_\alpha$ is constant, and let
$g(\alpha)$ be this constant value.

Now suppose that the lemma holds for $n \geq 1$ and we prove it for $n+1$. Thus let $f: [\kappa]^{n+1} \to 2$. For each $\langle \nu_0, \cdots, \nu_{n-1}  \rangle \in [\kappa]^n$, let $f_{\langle \nu_0, \cdots, \nu_{n-1}  \rangle}$ be defined by
\[
f_{\langle \nu_0, \cdots, \nu_{n-1}  \rangle}(\nu) = f(\langle \nu_0, , \cdots, \nu_{n-1}, \nu  \rangle).
\]
By the induction hypothesis, we can find sets $A^{\langle \nu_0, \cdots, \nu_{n-1}  \rangle}_\alpha \in U_\alpha, \alpha < \delta$, and a function $g_{\langle \nu_0, \cdots, \nu_{n-1}  \rangle}: \delta \to \tau$ such that
whenever $\alpha < \delta,$ then for all $\nu \in A^{\langle \nu_0, \cdots, \nu_{n-1}  \rangle}_\alpha $,
\[
f(\langle \nu_0, \cdots, \nu_{n-1}, \nu  \rangle)= g_{\langle \nu_0, \cdots, \nu_{n-1}  \rangle}(\langle \alpha  \rangle).
\]
Define $H: [\kappa]^n \to$$^{\delta}\tau$ by $H(\langle \nu_0, \cdots, \nu_{n-1}  \rangle):=g_{\langle \nu_0, \cdots, \nu_{n-1}  \rangle}$.
By the induction hypothesis, we can find sets $B_\alpha \in U_\alpha$ and a function $G: [\delta]^n \to$$^{\delta}\tau$ such that
 for each finite sequence $\langle \alpha_0,\cdots, \alpha_{n-1}  \rangle$ of ordinals less than $\delta$ and all increasing sequences $\langle \nu_0, \cdots, \nu_{n-1}  \rangle \in \prod_{i\leq n-1}B_{\alpha_i}$, we have
\[
g_{\langle \nu_0, \cdots, \nu_{n-1}  \rangle}=G(\langle \alpha_0,\cdots, \alpha_{n-1}  \rangle).
\]
This gives us a definable function $g: [\delta]^{n+1} \to \tau$, defined by
\[
g(\langle \alpha_0, \cdots, \alpha_{n}  \rangle)=g_{\langle \nu_0, \cdots, \nu_{n-1}  \rangle}(\langle \alpha_n \rangle),
\]
for some (and hence any) increasing sequence $\langle \nu_0, \cdots, \nu_{n-1}  \rangle \in \prod_{i\leq n-1} B_{\alpha_i} $. Now let $A_\alpha := B_\alpha \cap \triangle_{\langle \nu_0, \cdots, \nu_{n-1}  \rangle \in [\kappa]^n}A^{\langle \nu_0, \cdots, \nu_{n-1}  \rangle}_\alpha \in U_\alpha.$ Let $\langle \alpha_0,\cdots, \alpha_{n}  \rangle$ be a finite set of ordinals less than $\delta$ and let $\langle \nu_0, \cdots, \nu_{n}  \rangle \in \prod_{i\leq n}A_{\alpha_i}$ be an increasing sequence. Then $\nu_n \in A^{\langle \nu_0, \cdots, \nu_{n-1}  \rangle}_{\alpha_n}$, and we have

\[
f(\langle \nu_0, \cdots, \nu_{n}  \rangle)=g_{\langle \nu_0,\cdots, \nu_{n-1}  \rangle}(\alpha_n)=g(\langle \alpha_0,\cdots, \alpha_{n}  \rangle),
\]
as required.
\end{proof}
Suppose $p= \langle \kappa, A  \rangle \in \mathbb{M}$ and $c: [A]^{<\omega} \to \tau,$ where $\tau < \kappa.$ By shrinking $A$, we may assume that $A= \biguplus_{\alpha < \delta}A(\alpha),$
where $A(\alpha) \in \mathcal{U}(\kappa, \alpha)$. Then we can apply the above lemma and find sets $B(\alpha)\in \mathcal{U}(\kappa, \alpha), B(\alpha) \subseteq A(\alpha),$
such that  whenever $\langle \alpha_0,\cdots, \alpha_{n-1}  \rangle$
is a finite sequence of ordinals less than $\delta,$ the function $c$ is constant on the set of increasing sequences $\langle \nu_0, \cdots, \nu_{n-1}  \rangle \in\prod_{i\leq n-1} B(\alpha_i)$. Note that $B= \biguplus_{\alpha < \delta}B(\alpha) \in \mathcal{F}(\kappa)$ and hence $q= \langle \kappa, B  \rangle \in \mathbb{M}$ and is an extension of $p$.
 \begin{lemma}[Generalized R\"{o}wbottom's Lemma]\label{RowbottomMagidor}
Let $p\in\mathbb{M}$. For each function $c:[\biguplus_{i\leq n^p} A^p_i]^{<\omega}\rightarrow \tau$, where $\tau < \alpha_0^p,$ there is a sequence $\langle B_i\mid i\leq n^p\rangle$  which is homogeneous for $c$. Here homogeneity means the following:\footnote{For simplicity we will require in the notion of homogeneity that (1) holds, though this is not strictly necessary. }
\begin{enumerate}
\item for each $i\leq n^p$, $B_i\s A^p_i$, $B_i\in\mathcal{F}(\alpha^p_i)$ and $B_i=\biguplus_{\alpha < \delta}B_i(\alpha)$, for some
$B_i(\alpha) \in \mathcal{U}(\alpha^p_i, \alpha)$;
%\item for each $i\leq n^p$, there exists some $g_i: [\delta]^{<\omega} \to \tau;$
\item for each $m<\omega$ and $\vec{x}, \vec{y}\in [\biguplus_{i\leq n^p} B_i]^m$, if for  $k<m,$  $\vec{x}(k), \vec{y}(k)$
belong to the same $B_i(\alpha)$, for some $i\leq n^p$ and $\alpha < \delta$, then $c(\vec{x}) =c(\vec{y}).$
\end{enumerate}
\end{lemma}

\begin{proof}
Let us argue by induction over the length of $p$ and over the coherent sequences of measures. If $n^p=1$ the argument is covered by Lemma
\ref{RowbottomMagidor1}, so let  us suppose that $n^p>1$ and that the result holds for all conditions of length less than $n^p$.

Set $n:= n^p$ and say
$p=\langle \langle \alpha_0, A_0\rangle, \cdots, \langle \alpha_n, A_n \rangle    \rangle$.
Fix $\tau<\alpha_0$ and \linebreak$c\colon[\biguplus_{i\leq n} A_i]^{<\omega}\rightarrow \tau$ a function. For a sequence $\vec y\in [\biguplus_{i\leq n-1} A_i]^{<\omega}$, define $c_{\vec{y}}\colon [A_n]^{<\omega}\rightarrow\tau$ by $
c_{\vec{y}}(\vec{x}):=c(\vec y ^{\frown} \vec x).$
Arguing as in the base case we can find $A^{\vec y} \subseteq A_n$
witnessing clauses (1) and (2) for $c_{\vec{y}}$.  In particular, for each such $\vec y$, we can find a func\-tion
$g_{\vec y}: [\delta]^{< \omega} \to \tau$ such that for each $\vec{\alpha}\in [\delta]^{<\omega}$
and all increasing sequences $\vec{x} \in \prod A^{\vec y}(\vec{\alpha})$, $c(\vec y^{\frown} \vec{x})=g_{\vec y}(\vec{\alpha}).$ Set $$B_n:=\bigcap \{ A^{\vec y}\mid     \vec y\in [\biguplus_{i\leq n-1} A_i]^{<\omega} \}.$$

Define $d$ on $[\biguplus_{i\leq n-1} A_i]^{<\omega}$ by $d(\vec y):=g_{\vec y}$. As $\tau^{\delta^{<\omega}} < \alpha_0$, the induction hypothesis give us a
 sequence $\langle B_i\mid  i \leq n-1   \rangle$ of sets  witnessing clauses (1) and (2) with respect to $d$.

\begin{claim}
$\langle B_i\mid i\leq n\rangle$  witnesses clause (1) and (2) for $c$.
\end{claim}
\begin{proof}[Proof of claim]
We are left with checking that clause (2) is holds.
Suppose that  $m<\omega$, $\vec{z}_1, \vec{z}_2\in [\biguplus_{i\leq n} B_i]^m$ and  that for each $k<m,$  $\vec{z}_1(k), \vec{z}_2(k)$
belong to the same $B_i(\alpha)$, for some $i\leq n$ and $\alpha < \delta$. Then we can find $\vec{x}_1, \vec{x}_2 \in [B_n]^{\leq m}$
and  $\vec{y}_1, \vec{y}_2 \in [\biguplus_{i \leq n-1} B_i]^{\leq m}$ with $|\vec{x}_1|=|\vec{x}_2|$ and $|\vec{y}_1|=|\vec{y}_2|$
such that $\vec{z}_1=\vec y_1 ^{\frown} \vec x_1$ and
$\vec{z}_2=\vec y_2 ^{\frown} \vec x_2$.
By the choice of $\vec y_1$ and $\vec y_2$ and homogeneity of the sequence $\langle B_i\mid  i \leq n-1   \rangle$ for $d$, $d(\vec y_1)=d(\vec y_2)$ which means $g_{\vec y_1}=g_{\vec y_2}$. Similarly our choice of $\vec x_1, \vec x_2$ and the homogeneity of $B_n$ yields
$
c(\vec z_1)=c(\vec y_1 ^{\frown} \vec x_1)=g_{\vec y_1}(\vec x_1)=g_{\vec y_2}(\vec x_2)=c(\vec y_2 ^{\frown} \vec x_2)=c(\vec z_2),$ which gives the desired result.
\end{proof}
The above claim finishes the proof of the induction step and thus yields the lemma.
\end{proof}
\section{Proof of Theorem \ref{main theorem 1}}\label{ProofTheorem1}
We will devote the current section to present the main forcing construction used in the  proof of Theorem \ref{main theorem 1} in case $\Theta=\lambda^+$. Hereafter, $\delta,\kappa, \lambda$ will be assumed as in the statement of Theorem \ref{main theorem 1} and $G\s\Add(\kappa,\lambda^+)$ will be a fixed generic filter  over $V$.

\begin{notation}\label{CohenForcingNotationMagidor}
\rm{$ $
\begin{itemize}
\item For each $x\subseteq \lambda^+$,  $\mathbb{A}_x:=(\Add(\kappa,x),\supseteq)$. %will denote the forcing $\Add(\kappa, x)$: i.e., the set  the set of partial functions $p: \kappa\times x\rightarrow 2$ with $\dom(p)\in[\kappa\times x]^{<\kappa}$, endowed with the reverse end-extension ordering.
\item For each $y\subseteq x\subseteq\lambda^+
$ and $H\s\mathbb{A}_x$ a generic filter over $V$, $H\upharpoonright y$ will denote the generic filter induced by $H$ and the standard projection between $\mathbb{A}_x$ and $\mathbb{A}_y$.
\end{itemize}
}
\end{notation}
By a result of Woodin (see e.g. \cite[\S2]{GitShe}) it is feasible to prepare the ground model and make the strongness of $\kappa$ indestructible under adding arbitrary many Cohen subsets of $\kappa$. Thus we may assume that $\kappa$ is strong in $V[G]$. The following  property of strong cardinals is standard \cite{Kan}.
\begin{prop}
For a $(\kappa+2)$-strong cardinal $\kappa$, $o(\kappa)=(2^\kappa)^+$. Thus, $(\kappa+2)$-strong cardinals $\kappa$ have maximal Mitchell-order.
\end{prop}
By virtue of the previous proposition, in $V[G]$, $\kappa$ is a measurable cardinal with $o(\kappa)=\delta$. Thus, we may fix in $V[G]$ a coherent sequence of measures $\mathcal{U}=\langle\mathcal{U}(\alpha,\beta)\mid \alpha\leq\kappa, \beta<o^{\mathcal{U}}(\alpha)\rangle$ with $o(\kappa)=o^{\mathcal{U}}(\kappa)=\delta$. For each pair $(\alpha,\beta)\in\dom(\mathcal{U})$, in what follows $\dot{\mathcal{U}}(\alpha,\beta)$ will stand for $\mathbb{A}_{\lambda^+}$-name for the measure $\mathcal{U}(\alpha,\beta)$.
%\begin{notation}$ $
%\begin{enumerate}
%\item $D:=\{X \subseteq \lambda\mid  \text{$\sup(X)=\lambda$ and $X$  is ${>}\kappa$-closed}\}$,\footnote{Recall that a set $X$ of ordinals is $(> \kappa)$-closed if it contains its limit points of cofinality $> \kappa.$}
%\item $D^+:=\{X \subseteq \lambda^+\mid \text{$\sup(X)=\lambda^+$ and $X$ is ${>}\kappa$-closed} \}.$
%\end{enumerate}

%\end{notation}
%The following is a key lemma for the definition of the main forcing:
\begin{lemma}\label{LemmaD+}
There exists an  unbounded set of ordinals $\mathcal{A}\subseteq \lambda^+$, closed under taking limits of ${\geq}\kappa^+$-se\-quences, such that, for every  $\xi\in \mathcal{A}$ and every $\mathbb{A}_{\lambda^+}$-generic filter $G$,
$\mathcal{U}_{\xi}:=\langle\dot{\mathcal{U}}(\alpha,\beta)_{G}\cap V[G\upharpoonright\xi]\mid ~\alpha\leq\kappa,\, \beta<o^{\dot{\mathcal{U}}}(\alpha)\rangle$
 is a coherent sequence of measures in $V[G\upharpoonright\xi]$.
\end{lemma}
\begin{proof}
Arguing as in \cite[Lemma 3.3]{FriHon}, for each $(\alpha,\beta)\in\dom(\mathcal{U})$, there is an unbounded set $\mathcal{A}_{(\alpha,\beta)}\s \lambda^+$, closed under taking limits of ${\geq}\kappa^+$-sequences for which  $\dot{\mathcal{U}}(\alpha,\beta)_G\cap V[G\upharpoonright\xi]$ is a normal measure on $\alpha$ in $V[G\upharpoonright\xi]$, for all $\xi\in\mathcal{A}_{(\alpha,\beta)}$. Clearly, the collection of unbounded sets in $\lambda^+$ which are closed under taking limits of ${\geq}\kappa^+$-sequences is closed under taking intersections of $\kappa$-many sets. Set $\mathcal{A}:=\bigcap_{(\alpha,\beta)\in\dom(\mathcal{U})}\mathcal{A}_{(\alpha,\beta)}$ and observe that $\mathcal{A}\s\lambda^+$ is  unbounded and closed in the above mentioned sense. For each $\xi\in \mathcal{A}$, set  $\mathcal{U}_\xi:=\langle \mathcal{U}_\xi(\alpha,\beta)\mid \alpha\leq \kappa,\, \beta<o^\mathcal{U}(\alpha)\rangle$, where $(\alpha,\beta)\in\dom(\mathcal{U})$ and $\mathcal{U}_\xi(\alpha,\beta):=\dot{\mathcal{U}}(\alpha,\beta)_G\cap V[G\upharpoonright\xi]$.

We claim that $\mathcal{A}$ is as desired. Fix $\xi\in \mathcal{A}$ and let $(\alpha,\beta)\in\dom(\mathcal{U})$. By construction $\mathcal{U}_\xi(\alpha,\beta)$ is a normal measure on $\alpha$, hence (1) of Definition \ref{CoherentSequence} holds. Let $i^{\alpha,\xi}_\beta$ be the ultrapower by $\mathcal{U}_\xi(\alpha,\beta)$ in $V[G\upharpoonright\xi]$. We are left with showing that $\mathcal{U}_\xi$ satisfies clause (2) of Definition \ref{CoherentSequence}.
\begin{claim}
$i^{\alpha,\xi}_\beta(\mathcal{U}_\xi)(\varrho,\nu) = \mathcal{U}_\xi(\varrho,\nu)$, for each $(\varrho,\nu)\in \dom(\mathcal{U}_\xi\upharpoonright(\alpha,\beta)).$
\end{claim}

\begin{proof}[Proof of claim]
 Let $(\varrho,\nu)\in\dom(\mathcal{U}_\xi\upharpoonright(\alpha,\beta))$.  By definition, $(\varrho,\nu)$ is member of $ \dom(\mathcal{U}\upharpoonright(\alpha,\beta))$. We now check that the claim holds. Let us distinguish two cases:  $\varrho<\alpha$ and $\varrho=\alpha$.

 $\blacktriangleright$ Assume $\varrho<\alpha$. If $X\in\mathcal{P}(\varrho)\cap V[G\upharpoonright\xi]$, observe that $i^{\alpha,\xi}_\beta(X)=X$, $i^{\alpha,\xi}_\beta(\varrho)=\varrho$ and $i^{\alpha,\xi}_\beta(\nu)=\nu$. Then, it is not hard to check that the desired equality holds.

 $\blacktriangleright$ Assume $\varrho=\alpha$. Then, $\nu<\beta$. By coherence of $\mathcal{U}$, $\mathcal{U}_\xi(\alpha,\nu)=j^{\alpha}_\beta(\mathcal{U})(\alpha,\nu)\cap V[G\upharpoonright\xi]$.
 Let $X\in \mathcal{U}_\xi(\alpha,\nu)$ and
$f\in V[G\upharpoonright\xi]$ be such that $X=[f]_{\mathcal{U}_\xi(\alpha,\beta)}$. Since $f\in V[G\upharpoonright\xi]$, it is not hard to check that $[f]_{\mathcal{U}_\xi(\alpha,\beta)}=[f]_{\mathcal{U}(\alpha,\beta)}$.

  Notice that $\Ult(V[G], \mathcal{U}(\alpha,\beta))\models \text{$``X\in j^{\alpha}_\beta(\mathcal{U})(\alpha,\nu)$''}$ and that this is true if and only if $Y:=\{\delta<\alpha\mid f(\delta)\in \mathcal{U}(\delta,\nu)\}\in \mathcal{U}(\alpha,\beta) .$
Since $f\in V[G\upharpoonright\xi]$, $f(\delta)\in \mathcal{U}_\xi(\delta,\nu)$ iff $f(\delta)\in \mathcal{U}(\delta,\nu)$, hence $Y=\{\delta<\alpha\mid f(\delta)\in \mathcal{U}_\xi(\delta,\nu)\}\in\mathcal{U}(\alpha,\beta)$. Observe that $Y\in V[G\upharpoonright\xi]$, hence the above is equivalent to  $Y\in \mathcal{U}_\xi(\alpha,\beta)$. Thus $[f]_{\mathcal{U}_\xi(\alpha,\beta)}\in i^{\alpha,\xi}_\beta(\mathcal{U}_\xi)(\alpha,\nu)$. Combining all the previous equivalences we arrive at
 $i^{\alpha,\xi}_\beta(\mathcal{U}_\xi)(\alpha,\nu)=j^\alpha_\beta(\mathcal{U})(\alpha,\nu)\cap V[G\upharpoonright\xi]=\mathcal{U}_\xi(\alpha,\nu),$
 as desired.
\end{proof}

\begin{claim}
$\dom(i^{\alpha,\xi}_\beta(\mathcal{U}_\xi)\upharpoonright \alpha+1)=\dom (\mathcal{U}_\xi\upharpoonright(\alpha,\beta))$.
\end{claim}

\begin{proof}[Proof of claim]
The above argument already gives the right to left inclusion. Let $(\varrho,\nu)\in \dom(i^{\alpha,\xi}_\beta(\mathcal{U}_\xi)\upharpoonright \alpha+1)$. It is not hard to check that if $\varrho<\alpha$ then $(\varrho,\nu)\in \dom (\mathcal{U}_\xi\upharpoonright(\alpha,\beta))$, so that we may assume  $\varrho=\alpha$. We have to show that $o^{i^{\alpha,\xi}_\beta(\mathcal{U}_\xi)}(\alpha)=\beta$.
Let
\[
Y=\{\varrho<\alpha \mid o^{\mathcal{U}}(\varrho)=\beta  \}.
\]
Since $\mathcal{U}$ is a coherent sequence of measures, $o^{j^\alpha_\beta(\mathcal{U})}(\alpha)=\beta$, i.e., $\alpha \in j^\alpha_\beta(Y)$,
and hence $Y \in \mathcal{U}(\alpha, \beta)$. Since $\mathcal{U} \restriction \alpha=\mathcal{U}_\xi \restriction \alpha,$ we have $Y \in V[G \restriction \xi]$ and hence $Y \in \mathcal{U}_\xi(\alpha, \beta)$. Thus  $\alpha \in i^{\alpha, \xi}_\beta(Y)$, which means
$o^{i^{\alpha,\xi}_\beta(\mathcal{U}_\xi)}(\alpha)=\beta$, as required.
% hence $\nu<\beta$ and thus $(\alpha,\nu)\in \dom(\mathcal{U}_\xi\upharpoonright (\alpha,\beta))$.
%Arguing as above, $i^{\alpha,\xi}_\beta(\mathcal{U}_\xi)(\alpha,\nu)=j^{\alpha}_\beta(\mathcal{U})(\alpha,\nu)\cap V[G\upharpoonright\xi]$, hence $(\alpha,\nu)\in \dom(j^\alpha_\beta(\mathcal{U})\upharpoonright\alpha+1)$. By coherence of $\mathcal{U}$ it thus follows that $(\alpha,\nu)\in \dom(\mathcal{U}_\xi\upharpoonright(\alpha,\beta))$, as wanted.
\end{proof}
The above claims yield $i^{\alpha,\xi}_\beta(\mathcal{U}_\xi)\upharpoonright\alpha+1=\mathcal{U}_\xi\upharpoonright(\alpha,\beta)$, thus completing the proof of the lemma.
\end{proof}
\begin{remark}
\rm{For each $\xi\in \mathcal{A}$ and $\alpha<\kappa$, observe that $\mathcal{U}_\xi(\alpha,\beta)=\mathcal{U}(\alpha,\beta)$ and $\mathcal{U}(\alpha,\beta)\in V$, as $\mathbb{A}_\xi$ does not add bounded subsets of $\kappa$.}
\end{remark}

Let $\mathcal{A}$ be a set given by Lemma \ref{LemmaD+}. Hereafter we will be relying on the following notation:
\begin{notation}
\rm{For each $\xi\in\mathcal{A}$, let $\mathcal{U}_\xi$ %and $\mathfrak{F}_\alpha$
be the coherent sequence of measures resulting of  Lemma \ref{LemmaD+} and let $\dot{\mathcal{U}}_\xi$ be a $\mathbb{A}_\xi$-name such that $\mathcal{U}_\xi=(\dot{\mathcal{U}}_\xi)_{G\upharpoonright\xi}$. Similarly, $\dot{\mathbb{M}}_\xi$ will be a $\mathbb{A}_\xi$-name such that $\mathbb{M}_{\mathcal{U}_\xi}=(\dot{\mathbb{M}}_\xi)_{G\upharpoonright \xi}$. By convention, $\mathcal{U}_{\lambda^+}:=\mathcal{U}$ and $\dot{\mathbb{M}}_{\lambda^+}$ will be a $\mathbb{A}_{\lambda^+}$-name such that  $\mathbb{M}_{\mathcal{U}_{\lambda^+}}=(\dot{\mathbb{M}}_{\mathcal{U}_{\lambda^+}})_G$. %defined with respect to the tuple $(\kappa,\mu,\mathfrak{U}_\alpha, \mathfrak{B}_{\alpha})$.
}
\end{notation}

\begin{prop}\label{ProjectionGenericsinMagidor}
Work in $V$. For each $\xi\in\mathcal{A}$, $\mathbb{A}_{\lambda^+} \ast \mathbb{M}_{\lambda^+}$ projects onto $\mathbb{A}_{\xi} \ast\mathbb{M}_\xi$.
\end{prop}
\begin{proof}
Let $\xi\in\mathcal{A}$. It suffices to prove that any $\mathbb{M}_{\lambda^+}$-generic over $V[G]$ induces a $\mathbb{M}_\xi$-generic over $V[G\upharpoonright\xi]$.  Notice that by the discussion carried out at the end of Section \ref{Preliminaries} it suffices  with showing that any Magidor sequence $\vec{\gamma}$ for $\mathbb{M}_{\lambda^+}$ over $V[G]$ is also a  Magidor sequence for $\mathbb{M}_{\xi}$ over $V[G\upharpoonright\xi]$. To this aim we will check that $\vec{\gamma}$ witnesses (1) and (2) of  Theorem \ref{MitchellCriterion}, when $\mathcal{U}=\mathcal{U}_\xi$ and $V=V[G\upharpoonright\xi]$.

\begin{claim}
For each $\alpha<|\vec{\gamma}|$, $\vec{\gamma}\upharpoonright\alpha$ is a Magidor sequence for  $\mathbb{M}_{\mathcal{U}_\xi\upharpoonright\gamma(\alpha)+1}$ over $V[G\upharpoonright\xi]$
\end{claim}
\begin{proof}[Proof of claim]
Let us argue this by induction on $\alpha<|\vec{\gamma}|$. Assume that for each $\beta<\alpha$, $\vec{\gamma}\upharpoonright\beta$ is a Magidor sequence for  $\mathbb{M}_{\mathcal{U}_\xi\upharpoonright(\vec{\gamma}\upharpoonright \beta+1)}$ over $V[G\upharpoonright\xi]$. In order to check that $\vec{\gamma}\upharpoonright\alpha$ witnesses the inductive step we need to verify that (1) and (2) of Theorem \ref{MitchellCriterion} hold. Clearly our induction assumption yields (1), hence we are left with verifying (2). Fix $X\in\mathcal{P}(\kappa)^{V[G\upharpoonright\xi]}$. First let us assume that $X\in\mathcal{F}_{\mathcal{U}_\xi}(\kappa)^{V[G\upharpoonright\xi]}$. Since $\mathcal{F}_{\mathcal{U}_\xi}(\kappa)^{V[G\upharpoonright\xi]}\s\mathcal{F}_{\mathcal{U}_{\lambda^+}}(\kappa)^{V[G]}$ and $\vec{\gamma}$ is a Magidor sequence for $\mathbb{M}_{\lambda^+}$, $\vec{\gamma}(\sigma)\in X$,  for a tail end of $\sigma<|\vec{\gamma}|$. Conversely, assume that for a tail end of $\sigma<|\vec{\gamma}|$, $\vec{\gamma}(\sigma)\in X$. As before, $X\in\mathcal{F}_{\mathcal{U}_{\lambda^+}}(\kappa)^{V[G]}$. However, notice that $\mathcal{F}_{\mathcal{U}_\xi}(\kappa)^{V[G\upharpoonright\xi]}=\mathcal{F}_{\mathcal{U}_{\lambda^+}}(\kappa)^{V[G]}\cap \mathcal{P}(\kappa)^{V[G\upharpoonright\xi]}$, hence $X\in \mathcal{F}_{\mathcal{U}_\xi}(\kappa)^{V[G\upharpoonright\xi]}$, as wanted
\end{proof}
The verification of Theorem \ref{MitchellCriterion} (2) for $\vec{\gamma}$ is essentially contained in the proof of the above claim. It thus follows that $\vec{\gamma}$ is as desired.

\end{proof}

Fix $\xi_0\in\mathcal{A}\setminus \lambda+1$ and $\pi:\xi_0\rightarrow \mathrm{Even}(\lambda)$ be a bijection.\footnote{For an ordinal ${\alpha}$, $\mathrm{Even}(\alpha)$ stands for the set of all even and limit ordinals ${\leq} \alpha$. } Hereafter, $\xi_0$ will be fixed. The particular choice of this ordinal is not relevant, we could just have taken any other in $\mathcal{A}\setminus \lambda+1$. Evidently, $\pi$ entails an $\in$-isomorphism between the generic extensions $V^{\mathbb{A}_{\xi_0}}$ and $V^{\mathbb{A}_{\mathrm{Even}(\lambda)}}$. Thus, defining $\dot{\mathcal{U}}^\pi_{\xi_0}:= \pi(\dot{\mathcal{U}}_{\xi_0})$, $(\dot{\mathcal{U}}^\pi_{\xi_0})_{\pi[G\upharpoonright\xi_0]}=(\dot{\mathcal{U}}_{\xi_0})_{G\upharpoonright\xi_0}=\mathcal{U}_{\xi_0}$.  Say that $\mathcal{U}^\pi_{\xi_0}(\alpha,\beta)$
are the measures of this sequence. For the enlighten of the notation, let  $H$ be the $\mathbb{A}_{\mathrm{Even}(\lambda)}$-generic filter generated by $\pi[G\upharpoonright\xi_0]$.
\begin{lemma}\label{ReflectionMeasuresD}
There exists an  unbounded set of ordinals $\mathcal{B}\subseteq \lambda$, closed under taking limits of ${\geq}\kappa^+$-se\-quences, such that, for every  $\gamma\in \mathcal{B}$ and every $\mathbb{A}_{\mathrm{Even}(\lambda)}$-generic filter $H$,
$\mathcal{U}^\pi_{\gamma}:=\langle\dot{\mathcal{U}}^\pi_{\xi_0}(\alpha,\beta)_{H}\cap V[H\upharpoonright\mathrm{Even}(\gamma)]\mid \alpha\leq\kappa,\, \beta<o^{\dot{\mathcal{U}}}(\alpha)\rangle$
 is a coherent sequence of measures in $V[H\upharpoonright\mathrm{Even}(\gamma)]$.
\end{lemma}
\begin{proof}
Similar to the proof of  Lemma \ref{LemmaD+}.
\end{proof}
\begin{notation}\label{NotationUpi}
\rm{
For each $\gamma\in\mathcal{B}$, let  $\mathcal{U}^\pi_\gamma$ denote the coherent sequence of measure witnessing Lemma \ref{ReflectionMeasuresD}. By convention, $\mathcal{U}^\pi_\lambda:=\mathcal{U}_{\xi_0}$.
For each $\gamma\in\mathcal{B}\cup \{\lambda\}$,  let $\dot{\mathbb{M}}^\pi_\gamma$ be a $\mathbb{A}_{\mathrm{Even}(\gamma)}$-name for the Magidor forcing $\mathbb{M}_{\mathcal{U}^\pi_\gamma}$ in the generic extension $V[H\upharpoonright\mathrm{Even}(\gamma)]$. %defined with respect to the tuple $(\kappa,\mu,\mathfrak{U}^\pi_\alpha, \mathfrak{B}^\pi_{\alpha})$. %, where $\mathfrak{F}^\pi_{\alpha}=\pi''\mathfrak{F}_{\beta_0}\cap V[H\upharpoonright\mathrm{Even}(\alpha)]$ and $\mathfrak{B}^\pi_\alpha=\pi'' \mathfrak{B}_{\beta_0}\cap V[H\upharpoonright\mathrm{Even}(\alpha) ]$.
}
\end{notation}
\begin{lemma}\label{ProjectionsCohenPartMagidor}
Let $\hat{\mathcal{A}}=(\mathcal{A}\cap (\xi_0,\lambda^+))\cup\{\lambda^+\}$.
\begin{enumerate}
\item For every $\xi,\tilde{\xi}\in\hat{\mathcal{A}}$ with $\xi<\tilde{\xi}$, there is a projection
$$
\sigma^{\tilde{\xi}}_\xi: \mathbb{A}_{\tilde{\xi}}\ast \dot{\mathbb{M}}_{\tilde{\xi}}\rightarrow \mathrm{RO}^+(\mathbb{A}_{\xi}\ast \dot{\mathbb{M}}_\xi).
$$
\item For every $\xi\in\hat{\mathcal{A}}$ and $\gamma\in\mathcal{B}$, there is a projection
$$
\sigma^{\xi}_\gamma: \mathbb{A}_\xi\ast \dot{\mathbb{M}}_\xi\rightarrow \mathrm{RO}^+(\mathbb{A}_{\rm{Even}(\gamma)}\ast \dot{\mathbb{M}}^\pi_\gamma).
$$
\item For every $\xi\in\hat{\mathcal{A}}$ and $\gamma\in\mathcal{B}$, let $\hat{\sigma}^\xi_\gamma$ be the extension of $\sigma^\xi_\gamma$ to the Boolean completion of $\mathbb{A}_{\xi}\ast \dot{\mathbb{M}}_\xi$
$$
\hat{\sigma}^\xi_\gamma: \mathrm{RO}^+(\mathbb{A}_\xi\ast \dot{\mathbb{M}}_\xi)\rightarrow \mathrm{RO}^+(\mathbb{A}_{ \mathrm{Even}(\gamma)}\ast \dot{\mathbb{M}}^\pi_\gamma).
$$
Then the projections commute with $\sigma^{\lambda^{+}}_\gamma$:
$$
\sigma^{\lambda^{+}}_\gamma=\hat{\sigma}^{\xi}_\gamma\circ \sigma^{\lambda^{+}}_\xi.$$
\end{enumerate}
\end{lemma}
\begin{proof}
The argument for (3) is the same as in \cite[Lemma 3.8]{FriHon}.  Also,
(1) and (2) follow in the same fashion, so let us give details only for (1). Let $G\s \mathbb{A}_{\tilde{\xi}}$ generic over $V$ and $\vec{\gamma}$ be a Magidor sequence for $\mathbb{M}_{\tilde{\xi}}$ over $V[G]$. Appealing to Proposition \ref{ProjectionGenericsinMagidor} it is clear that $G\upharpoonright\xi \ast \dot{G}(\vec{\gamma})$ is $\mathbb{A}_\xi\ast \dot{\mathbb{M}}_\xi$-generic over $V$. Then any generic filter for $\mathbb{A}_{\tilde{\xi}}\ast \dot{\mathbb{M}}_{\tilde{\xi}}$ induces a generic filter for $\mathbb{A}_{\xi}\ast \dot{\mathbb{M}}_\xi$ and thus a generic filter for the Boolean completion $\mathrm{RO}^+(\mathbb{A}_{\xi}\ast \dot{\mathbb{M}}_\xi)$.
\end{proof}

\begin{defi}[Main forcing]\label{MainForcingRGolshani}
A condition in $\mathbb{R}$ is a triple $(p,\dot{q},r)$ for which all the following hold:
\begin{enumerate}
\item $(p,\dot{q})\in\mathbb{A}_{\lambda^+}\ast\dot{\mathbb{M}}_{\lambda^+}$;
\item $r$ is a partial function with $\dom(r)\in [\mathcal{B}]^{< \kappa^+}$;
\item For every $\gamma\in\dom(r)$, $r(\gamma)$ is a $\mathbb{A}_{\rm{Even}(\gamma)}\ast \dot{\mathbb{M}}^\pi_\gamma$-name such that
$$\one_{\mathbb{A}_{\rm{Even}(\gamma)}\ast \dot{\mathbb{M}}^\pi_\gamma}\Vdash_{\mathbb{A}_{\rm{Even}(\gamma)}\ast \dot{\mathbb{M}}^\pi_\gamma}\text{$``r(\gamma)\in \dot{\Add}(\kappa^+, 1)$''}.$$
\end{enumerate}
For conditions $(p_0,\dot{q}_0,r_0), (p_1,\dot{q}_1, r_1)$ in $\mathbb{R}$ we will write $(p_0,\dot{q}_0,r_0)\leq_\mathbb{R} (p_1,\dot{q}_1, r_1)$ iff $(p_0,\dot{q}_0)\leq_{\mathbb{A}_{\lambda^+}\ast \dot{\mathbb{M}}_{\lambda^+}}(p_1,\dot{q}_1)$, $\dom(r_1)\subseteq \dom(r_0)$ and for each $\gamma\in\dom(r_1)$,
$\sigma^{\lambda^{+}}_\gamma(p_0,q_0)\Vdash_{\mathbb{A}_{\rm{Even}(\gamma)}\ast \dot{\mathbb{M}}^\pi_\gamma}\text{$``r_0(\gamma)\leq r_1(\gamma)$''}.$
\end{defi}
\begin{defi}
$\mathbb{U}$ will denote the \emph{termspace forcing}. That is, the pair $(U,\leq)$ where $U:=\{(\one,\dot{\one},r)\mid (\one,\dot{\one},r)\in\mathbb{R}\}$ and $\leq$ is the  order inherited from $\mathbb{R}$. Set $\bar{\mathbb{R}}:=(\mathbb{A}_{\lambda^+}\ast\dot{\mathbb{M}}_{\lambda^+})\times \mathbb{U}$.
\end{defi}
%Here there are some properties of the above forcings:
\begin{prop}$ $\label{projectionUMagidor}
\begin{enumerate}
\item $\mathbb{U}$ is $\kappa^+$-directed closed.
\item The function $\rho: \bar{\mathbb{R}}\rightarrow \mathbb{R}$ given by $\langle(p,\dot{q}), (\one,\dot{\one},r)\rangle\mapsto (p,\dot{q},r)$ entails a projection. In particular,
$V^{\mathbb{A}_{\lambda^+}\ast \dot{\mathbb{M}}_{\lambda^+}}\subseteq V^\mathbb{R}\subseteq V^{\bar{\mathbb{R}}}.$
\item $V^{\mathbb{A}_{\lambda^+}\ast \mathbb{M}_{\lambda^+}}$ and $V^\mathbb{R}$ have the same ${<}\kappa^+$-sequences.
\end{enumerate}
\end{prop}
\begin{proof}
(1) Let $\langle (\one,\dot{\one},r_\alpha)\mid \alpha<\kappa\rangle$ be a $\leq_{\mathbb{R}}$-decreasing sequence of conditions in $\mathbb{U}$. Set $\dom(r^*):=\bigcup_{\alpha<\kappa}\dom(r_\alpha)$ and observe that $\dom(r^*)\in[\mathcal{B}]^{<\kappa^+}$. For each $\gamma\in\dom(r^*)$, let $\alpha_\gamma:=\min\{\alpha<\kappa\mid \gamma\in \dom(r_\alpha)\}$. Clearly, $\gamma\in\dom(r_\alpha)$, for each $\alpha\geq \alpha_\gamma$. Since our sequence is $\leq_{\mathbb{R}}$-decreasing this yields
$$\one_{\mathbb{A}_{\rm{Even}(\gamma)}\ast \dot{\mathbb{M}}^\pi_\gamma}\Vdash_{\mathbb{A}_{\rm{Even}(\gamma)}\ast \dot{\mathbb{M}}^\pi_\gamma}\text{$``\langle r_\alpha(\gamma)\mid  \alpha_\gamma\leq \alpha<\kappa\rangle$ is $\leq_
{\dot{\Add}(\kappa^+, 1)}$-decreasing''}.$$
Let $r^*(\gamma)$ be a $\mathbb{A}_{\rm{Even}(\gamma)}\ast \dot{\mathbb{M}}^\pi_\gamma$-name for a condition forced by $\one_{\mathbb{A}_{\rm{Even}(\gamma)}\ast \dot{\mathbb{M}}^\pi_\gamma}$ to be a lower bound for the above sequence. It is clear that $(\one,\dot{\one},r^*)$ provides the desired $\leq_{\mathbb{R}}$-lower bound.

(2) Clearly, $\rho$ is order preserving and $\rho(\one_{\bar{\mathbb{R}}})=\one_{\mathbb{R}}$. Let $(p_1,\dot{q}_1,r_1)\leq_{\mathbb{R}} \rho(\langle (p_2,\dot{q}_2), (\one,\dot{\one},r_2)\rangle)$. Define $r_3$ with $\dom(r_3)=\dom(r_1)$ such that, for each $\gamma\in\dom(r_3)$,
$\langle \sigma, b\rangle \in r_3(\gamma)$ iff the following hold: if $b\leq \sigma^{\lambda^+}_\gamma(p_1,\dot{q}_1)$ then $b\Vdash_{\mathbb{A}_{\rm{Even}(\gamma)\ast\dot{\mathbb{M}}^\pi_\gamma}} \sigma\in r_1(\gamma)$ and, otherwise, if $b\perp \sigma^{\lambda^+}_\gamma(p_1,\dot{q}_1)$, $b\Vdash_{\mathbb{A}_{\rm{Even}(\gamma)\ast\dot{\mathbb{M}}^\pi_\gamma}} \sigma\in r_2(\gamma)$. It is not hard to check that $ \langle (p_1,\dot{q}_1), (\one,\dot{\one},r_3)\rangle$ is $\leq_{\bar{\mathbb{R}}}$-below the condition $\langle (p_2,\dot{q}_2), (\one,\dot{\one},r_2)\rangle$ and $(p_1,\dot{q}_1,r_3)\leq_{\mathbb{R}}(p_1,\dot{q}_1,r_1)$. This shows that $\rho$ defines a projection and thus that $V^{\mathbb{R}}\s V^{\bar{\mathbb{R}}}$. The remaining inclusion follows from the trivial fact that $\mathbb{R}$ projects onto $\mathbb{A}_{\lambda^+}\ast\dot{\mathbb{M}}_{\lambda^+}$.

(3) Before proving the result we need to begin with an easy observation. Let $(p,\dot{q})\in \mathbb{A}_{\lambda^+}\ast \dot{\mathbb{M}}_{\lambda^+}$ and say that
$$p\Vdash_{\mathbb{A}_{\lambda^+}}\dot{q}=\langle \langle\tau_0,\dot{A}_0\rangle, \dots, \langle\tau_{n-1}, \dot{A}_{n-1}\rangle,\langle \check{\kappa},\dot{A}_n\rangle\rangle,$$
for $\tau_i, \dot{A}_i$ being $\mathbb{A}_{\lambda^+}$-names. Clearly, one may extend $p$ to a condition $p^*$ ensuring that, for each $i<n$, $p^*\Vdash_{\mathbb{A}_{\lambda^+}}\tau_i=\check{\alpha}_i$, for some $\alpha_i<\kappa$. In other words, the set of conditions of the form
$$\langle p, \langle \langle\check{\beta}_0,\dot{A}_0\rangle, \dots, \langle\check{\beta}_{n-1}, \dot{A}_{n-1}\rangle,\langle \check{\kappa},\dot{A}_n\rangle\rangle,$$
endowed with the induced order, forms a dense subposet of $\mathbb{A}_{\lambda^+}\ast \dot{\mathbb{M}}_{\lambda^+}$. Call this forcing $\mathbb{Q}$. Notice that $\mathbb{Q}$ and $\mathbb{A}_{\lambda^+}\ast \dot{\mathbb{M}}_{\lambda^+}$ are forcing equivalent, hence $V^{\mathbb{Q}}=V^{\mathbb{A}_{\lambda^+}\ast\dot{\mathbb{M}}_{\lambda^+}}$. By combining Proposition \ref{MagidorIsKnaster} with the $\kappa^+$-Knasterness of $\mathbb{A}_{\lambda^+}$ it is easy to show that $\mathbb{Q}$ is $\kappa^+$-Knaster. By Easton's Lemma (see e.g. \cite[Lemma 4.4.]{GolMoh}), $\one_{\mathbb{Q}}\Vdash_{\mathbb{Q}}\text{$``\mathbb{U}$ is $\kappa^+$-distributive''}$, hence $V^\mathbb{Q}$ and $V^{\mathbb{Q}\times\mathbb{U}}$ have the same ${<}\kappa^+$-sequences, thus $V^\mathbb{Q}$ and $V^{\bar{\mathbb{R}}}$ also. This latter assertion yields the desired result.
\end{proof}
Let $\bar{R}\subseteq \bar{\mathbb{R}}$ a generic filter whose projection onto $\mathbb{A}_{\lambda^+}$ generates the generic filter $G$. Also, let $R\subseteq \mathbb{R}$ be the generic filter generated by $\rho[\bar{R}]$ and $S\subseteq\mathbb{M}_{\lambda^+}$ be the generic filter over $V[G]$ induced by $\bar{R}$. We next prove some important properties of the forcing $\mathbb{R}$.
\begin{prop}[Some properties of $\mathbb{R}$]$ $\label{PropertiesOfV[R]Magidor}
\begin{enumerate}
\item $\mathbb{R}$ is $\lambda$-Knaster. In particular, all $V$-cardinals ${\geq}\lambda$ are preserved.
\item $\mathbb{R}$   preserves all the cardinals outside $((\kappa^+)^V,\lambda)$, while collap\-ses  the cardinals there to $(\kappa^+)^V$. In particular, $$V[R]\models \text{$``(\kappa^+)^V=\kappa^+\,\wedge\,\lambda=\kappa^{++}$''}.$$
\item $V[R]\models\text{$``2^\kappa= \lambda^{+}=\kappa^{+3}$''}$.
\item $V[R]\models \text{$``\kappa$ is strong limit with $\cof(\kappa)=\delta$''}$.
\end{enumerate}
\end{prop}

\begin{proof}
$ $
\begin{enumerate}
\item Let $A\in[\mathbb{R}]^\lambda$. By extending if necessary the conditions of $A$ we may further assume that $A$ is of the form $\{(p_\alpha,\dot{q}_\alpha, r_\alpha)\mid \alpha<\lambda\}$, where
$$p_\alpha\Vdash_{\mathbb{A}_{\lambda^+}}\dot{q}_\alpha=\langle(\check{\beta}^\alpha_0,\dot{A}^\alpha_0),\dots, (\check{\beta}^\alpha_{m_\alpha-1},\dot{A}^\alpha_{m_\alpha-1}),(\check{\kappa}, \dot{A}^\alpha_{m_\alpha}) \rangle.\footnote{See the discussion carried out in the proof Proposition \ref{projectionUMagidor} (3).} $$
Since  $\mathbb{A}_{\lambda^+}$ is $\kappa^+$-Knaster by passing to a set $\mathcal{I}\in[\lambda]^\lambda$ we may assume that, for all $\alpha,\gamma\in \mathcal{I}$, $p_\alpha\parallel p_\gamma$, $m^*=m_\alpha=m_\gamma$ and $\beta^\alpha_i=\beta^\gamma_i$, for $i<m^*$. Observe that for all $\alpha,\gamma\in\mathcal{I}$, $(p_\alpha\cup p_\gamma, \dot{q}_\alpha\wedge \dot{q}_\gamma)$ witnesses compatibility of $(p_\alpha,\dot{q}_\alpha)$ and $(p_\gamma,\dot{q}_\gamma)$.

On the other hand, appealing to the $\Delta$-system lemma \cite[Ch. 3, Lemma 6.15]{Kun}, we may refine $\mathcal{I}$ to $\mathcal{J}\in[\mathcal{I}]^\lambda$ and find $\Delta\in[\mathcal{B}]^{<\kappa^+}$ and $r^*$ in such a way that $\{\dom(r_\alpha)\mid\alpha\in\mathcal{J}\}$ forms a $\Delta$-system  with $r_\alpha\upharpoonright \Delta=r^*$, for $\alpha\in\mathcal{J}$. Indeed, this is feasible because  the set of $\bigcup_{\gamma\in\Delta}(\mathbb{A}_{\mathrm{Even}(\gamma)}\ast\dot{\mathbb{M}}^\pi_\gamma)$-names has cardinality less than $\lambda$. Altogether this shows that $\{p_\alpha\in A\mid \alpha\in\mathcal{J}\}$ is a subset of $A$ of pairwise compatible conditions with cardinality $\lambda$.

\item The preservation of cardinals ${\geq}\lambda$ is a consequence of item (1), so we are left with discussing what occurs with $V$-cardinals ${<}\lambda$. Let us begin arguing that cardinals ${\leq}(\kappa^+)^V$ are preserved. To this aim observe that $\mathbb{A}_{\lambda^+}\ast\dot{\mathbb{M}}_{\lambda^+}$ preserves cardinals ${\leq}(\kappa^+)^V$, hence Proposition~\ref{projectionUMagidor}(3) implies that $\mathbb{R}$ preserves cardinals ${\leq}(\kappa^+)^V$.\footnote{Actually, $\mathbb{A}_{\lambda^+}\ast\dot{\mathbb{M}}_{\lambda^+}$ is cardinal-preserving.}

Let us now argue that $\mathbb{R}$ collapses all $V$-cardinals in $(\kappa^+,\lambda)$.\footnote{Observe that there is no confusion between $\kappa^+$ and $(\kappa^+)^V$.}\linebreak Assuming this was true it is clear that they have to be collapsed to $\kappa^+$. For a $V$-cardinal $\theta\in (\kappa^+,\lambda)$,  let  $\eta_\theta:=\min\mathcal{B}\cap (\theta,\lambda)$. Notice that $\mathbb{R}$ projects onto $\mathrm{RO}^+(\mathbb{A}_{\mathrm{Even}(\eta_\theta)}\ast \dot{\mathbb{M}}^{\pi}_{\eta_\theta})\ast \dot{\Add}(\kappa^+,1)$ via the map $(p,\dot{q},r)\mapsto (\sigma^{\lambda^+}_{\eta_\theta}(p,\dot{q}), r(\eta_\theta))$, so if this latter forcing collapses $\theta$ then $\mathbb{R}$ also.
\begin{claim}
$\mathrm{RO}^+(\mathbb{A}_{\mathrm{Even}(\eta_\theta)}\ast \dot{\mathbb{M}}^{\pi}_{\eta_\theta})\ast \dot{\Add}(\kappa^+,1)$ collapses the interval $(\kappa^+, |\eta_\theta|]$. In particular, $\theta$ is collapsed.
\end{claim}
\begin{proof}[Proof of claim]
Observe that $\mathbb{A}_{\mathrm{Even}(\eta_\theta)}\ast \dot{\mathbb{M}}^{\pi}_{\eta_\theta}$ yields a generic extension where $2^\kappa\geq |\eta_\theta|$ and $(\kappa^+)^{\mathbb{A}_{\mathrm{Even}(\eta_\theta)}\ast \dot{\mathbb{M}}^{\pi}_{\eta_\theta}}=\kappa^+$. Working there, let $\langle f_\xi\mid \xi<|\eta_\theta|\rangle$ be an enumeration of $|\eta_\theta|$-many different  Cohen functions  added by this forcing. For each $\xi\in|\eta_\theta|$, set
$$D_\xi:=\{r\in\Add(\kappa^+,1)\mid \exists\zeta<\kappa^+\, \forall\gamma<\kappa\, f_\xi(\gamma)=p(\zeta+\gamma)\}.$$
It is fairly easy to check that $D_\xi$ is a dense subset of $\Add(\kappa^+,1)$. Let $A\s \Add(\kappa^+,1)$ generic over $V^{\mathbb{A}_{\mathrm{Even}(\eta_\theta)}\ast \dot{\mathbb{M}}^{\pi}_{\eta_\theta}}$ and define $\Phi:|\eta_\theta|\rightarrow \kappa^+$ by  $ \Phi(\xi):=\min\{\zeta<\kappa^+\mid \exists r\in A\, (\text{$\zeta$ witnesses that $r\in D_\xi$)}\} $. To prove the claim observe that it would suffice with showing that $\Phi$ is an injective function. For so, let $\xi\neq \xi'$ and assume that $\Phi(\xi)=\Phi(\xi')$. Denote this common value by $\sigma$. By definition, there are $r,s\in A$ such that, for all $\gamma<\kappa$, $f_\xi(\gamma)=r(\sigma+\gamma)$ and $f_{\xi'}(\gamma)=s(\sigma+\gamma)$ but, since $A$ is a filter, this entails $f_\xi=f_{\xi'}$, which yields a contradiction.
\end{proof}
\item By using the $\gch_{\geq\kappa}$ in the ground model, counting $\mathbb{A}_{\lambda^+}\ast \dot{\mathbb{M}}_{\lambda^+}$-nice names arguments yield $\text{$``2^\kappa=\lambda^+$''}$ in the corresponding generic extension. Now use Proposition \ref{projectionUMagidor}(3).
\item Once again, this follows by combing Proposition \ref{projectionUMagidor}(3) with the fact that $\mathbb{A}_{\lambda^+}\ast \dot{\mathbb{M}}_{\lambda^+}$ forces this property.
\end{enumerate}
\end{proof}

\section{$\TP(\kappa^{++})$ holds}\label{TPkappa++SectionMagidor}
In the present section we will prove that $V[R]\models \TP(\kappa^{++})$. For a more neat presentation we will simply give details in case $\Theta=\lambda^+$. The main ideas involved in the proof of the general case can be found in Section \ref{SectionGapArbitrary}.

\medskip

Let us briefly summarize the structure of the argument. First we beging proving that any counterexample for $\TP(\lambda)$ in $V[R]$ lies in an intermediate extension of $\mathbb{R}$. More formally, % we first define the notion of truncations of $\mathbb{R}$ and then we show that
any $\lambda$-Aronszajn tree in $V[R]$ is a $\lambda$-Aronszajn tree in a generic extension given by some truncation of $\mathbb{R}$ (see Proposition~\ref{AronszajnWeakCompMagidor}).  These truncations have the important feature that they are isomorphic to a Mitchell forcing $\mathbb{R}^*$ without  mismatches between the Cohen and the collapsing component.

In latter arguments we shall again consider truncations of $\mathbb{R}^*$, $\mathbb{R}^*\upharpoonright\xi$, and use the weak compactness of $\lambda$ to prove that any $\lambda$-Aronszajn tree in  $V^{\mathbb{R}^*}$ reflects to a $\xi$-Aronszajn tree in $V^{\mathbb{R}^*\upharpoonright\xi}$ (see Lemma \ref{AronszajnWeakCompMagidor}). Then, we will be in conditions to use Unger's ideas \cite{Ung} to show that there are no $\xi$-Aronszajn trees in $V^{\mathbb{R}^*\upharpoonright \xi}$, and thus that $V[R]\models \TP(\lambda)$. For the record of this section, recall that $\xi_0\in\mathcal{A}\setminus\lambda+1$ is the ordinal fixed in the previous section.

\begin{defi}[Truncations of $\mathbb{R}$]\label{truncationsR}
Let $\xi\in\mathcal{A}\cap (\xi_0,\lambda^+)$. % with $\beta_0<\alpha<\lambda^+$.
A condition in $\mathbb{R}\upharpoonright\xi$ is a triple $(p,\dot{q},r)$ for which all the following hold:
\begin{enumerate}
\item $(p,q)\in \mathbb{A}_\xi\ast \dot{\mathbb{M}}_\xi$;
\item $r$ is a partial function with $\dom(r)\in[\mathcal{B}]^{<\kappa^+}$;
\item For every $\alpha\in\dom(r)$, $r(\alpha)$ is a $\mathbb{A}_{\rm{Even}(\alpha)}\ast\dot{\mathbb{M}}^\pi_\alpha$-name such that
$$\one_{\mathbb{A}_{\mathrm{Even}(\alpha)}\ast \dot{\mathbb{M}}^\pi_\alpha}\Vdash_{\mathbb{A}_{\mathrm{Even}(\alpha)}\ast\dot{\mathbb{M}}^\pi_\alpha} \text{$``\dot{r}(\alpha)\in\dot{\Add}(\kappa^+,1)$''}.$$
\end{enumerate}
For conditions $(p_0,\dot{q}_0,r_0),(p_1,\dot{q}_1,r_1)$ in $\mathbb{R}\upharpoonright \xi$ we will write $(p_0,\dot{q}_0,r_0)\leq (p_1,\dot{q}_1,r_1)$ in case $(p_0,\dot{q}_0)\leq_{\mathbb{A}_{\mathrm{Even}(\alpha)}\ast\dot{\mathbb{M}}^\pi_\alpha} (p_1,\dot{q}_1)$, $\dom(r_1)$ $\subseteq\dom(r_0)$ and for each $\alpha\in\dom(r_1)$,
$\sigma^\alpha_\beta(p_0,q_0)\Vdash_{\mathbb{A}_{\mathrm{Even}(\alpha)}\ast \dot{\mathbb{M}}^\pi_\alpha}\text{$``\dot{r}_0(\alpha)\leq \dot{r}_1(\alpha)$''}. $
\end{defi}
The following is an immediate consequence of Lemma \ref{ProjectionsCohenPartMagidor}.
\begin{prop}\label{ProjectionRandTruncationsMagidor}
Let $\xi\in\mathcal{A}\cap (\xi_0,\lambda^+)$. Then there is a projection between $\mathbb{R}$ and $\mathrm{RO}^+(\mathbb{R}\upharpoonright\xi)$.
\end{prop}
\begin{prop}\label{TruncationsAndAronszajnTreeMagidor}
Let $\dot{T}$ be a $\mathbb{R}$-name for a $\lambda$-Aronszajn tree. There is $\xi^*\in\mathcal{A}\cap (\xi_0,\lambda^+)$, such that $V^{\mathbb{R}\upharpoonright\xi^*}\models \text{$``T$ is a $\lambda$-Aronszajn tree''}$ %$V^{\mathbb{R}\upharpoonright\beta^*}$.
\end{prop}
\begin{proof}
Let $\dot{T}$ be a $\mathbb{R}$-name for a $\lambda$-Aronszajn tree $T$. Without loss of generality we may assume $\one_\mathbb{R}\Vdash_\mathbb{R} \dot{T}\subseteq\check{\lambda}$. % and that $\dot{T}$ is a nice name for a subset of $\lambda$.
Let $\{{A}_\alpha\}_{\alpha<\lambda}$ be a family of maximal antichains deciding $\text{$``\check{\alpha}\in \dot{T}$''}$. Set %that the $\lambda$-Knaster property of $\mathbb{R}$ implies that
${A}^*:=\bigcup_{\alpha\in\lambda} A_\alpha$ and observe that $|A^*|\leq \lambda$. In particular, there is some $\xi^*\in\mathcal{A}\cap (\xi_0,\lambda^+)$ be such that  $\dom(p)\subseteq \kappa\times\xi^*$, for any condition $(p,\dot{q},r)\in A^*$. Clearly $\{A_\alpha\}_{\alpha<\lambda}$ is a family of maximal antichains in $\mathbb{R}\upharpoonright\xi^*$ deciding the same assertions, hence $V^{\mathbb{R}\upharpoonright\xi^*}\models \text{$``T$ is $\lambda$-Aronszajn''}$.
\end{proof}

Let $\pi^*:\xi^*\rightarrow \lambda$ be a bijection extending $\pi$. We use $\pi^*$ to define an $\in$-isomorphism between $V^{\mathbb{A}_{\xi^*}}$ and $V^{\mathbb{A}_\lambda}$.%and lift it to an isomorphism between $\mathbb{A}_{\beta^*}$ and $\mathbb{A}_\lambda$
\footnote{This choice will guarantee that our future construction coheres with the previous one.}
Again, $\mathcal{U}_\lambda^{\pi^*}:=(\pi^*{(\dot{\mathcal{U}}_{\xi^*})})_{\pi^*[G\upharpoonright\xi^*]}$ is a coherent sequence of measures %supercompact measure over $\mathcal{P}_\kappa(\gamma^+)^{V[\pi^*(G\upharpoonright \beta^*)]}$
which (pointwise) extends  $\mathcal{U}^\pi_{\lambda}$.
Let $\mathbb{M}^{\pi^*}_\lambda:=\mathbb{M}_{\mathcal{U}^{\pi^*}_\lambda}$. Let us denote by $\mathcal{U}_\lambda^{\pi^*}(\alpha,\beta)$ the measures appearing in $\mathcal{U}^{\pi^*}_\lambda$. For the ease of notation, let  $H^*$ be the $\mathbb{A}_{\mathrm{Even}(\lambda)}$-generic filter generated by $\pi^*[G\upharpoonright\xi^*]$.

\begin{prop}\label{R*RtruncatedAreIsomorphicMagidor} $ $
\begin{enumerate}
\item There is an isomorphism $\varphi:\mathbb{A}_{\xi^*}\ast \dot{\mathbb{M}}_{\xi^*}\rightarrow\mathbb{A}_{\lambda}\ast \dot{\mathbb{M}}^{\pi^*}_{\lambda}$.
\item For each $\xi\in\mathcal{B}$ the function $\varrho^{\lambda}_{\xi}=\sigma^{\xi^*}_{\xi}\circ \varphi^{-1}$ establishes a projection between $\mathbb{A}_{\lambda}\ast \dot{\mathbb{M}}^{\pi^*}_{\lambda}$ and $\mathrm{RO}^+(\mathbb{A}_{\even(\xi)}\ast \dot{\mathbb{M}}^{\pi}_\xi)$.
\end{enumerate}
\end{prop}
\begin{proof}
For (1), recall that the subposet of $\mathbb{A}_{\xi^*}\ast \dot{\mathbb{M}}_{\xi^*}$ with conditions of the form $s:=(p, \langle \langle\check{\beta}_0,\dot{A}_0\rangle, \dots, \langle\check{\beta}_{n-1}, \dot{A}_{n-1}\rangle,\langle \check{\kappa},\dot{A}_n\rangle\rangle)$ is dense. Analogously, the same is true for $\mathbb{A}_{\lambda}\ast \dot{\mathbb{M}}^{\pi^*}_{\lambda}$. It is now routine to check that the map   $s\mapsto(\pi^*(p),\langle \langle\check{\beta}_0,\dot{A}_0\rangle, \dots, \langle\check{\beta}_{n-1}, \dot{A}_{n-1}\rangle,\langle \check{\kappa},\dot{A}_n\rangle\rangle)$ defines an isomorphism\linebreak between these two dense subposets, which is enough to prove the desired result. Observe that now (2) is immediate as $\sigma^{\xi^*}_\xi$ is a projection.
\end{proof}

\begin{defi}\label{DefinitionR*Magidor}
A condition in $\mathbb{R}^*$ is a triple $(p,\dot{q},r)$ for which all the following hold:
\begin{enumerate}
\item $(p,q)\in \mathbb{A}_\lambda \ast \dot{\mathbb{M}}^{\pi^*}_\lambda$;
\item $r$ is a partial function with $\dom(r)\in[\mathcal{B}]^{<\kappa^+}$;
\item  For every $\xi\in\dom(r)$, $r(\xi)$ is an $\mathbb{A}_{\mathrm{Even}(\xi)}\ast \dot{\mathbb{M}}^\pi_\xi$-name such that
$$\one_{\mathbb{A}_{\mathrm{Even}(\xi)}\ast \dot{\mathbb{M}}^\pi_\xi}\Vdash_{\mathbb{A}_{\mathrm{Even}(\xi)}\ast \dot{\mathbb{M}}^\pi_\xi}\text{$``\dot{r}(\xi)\in\dot{\Add}(\kappa^+,1)$''}.$$
\end{enumerate}
For conditions $(p_0,\dot{q}_0,r_0),(p_1,\dot{q}_1,r_1)$  in $\mathbb{R}^*$ we will write $(p_0,\dot{q}_0,r_0)\leq (p_1,\dot{q}_1,r_1)$ in case $(p_0,\dot{q}_0)\leq_{\mathbb{A}_{\mathrm{Even}(\lambda)}\ast \dot{\mathbb{M}}^{\pi^*}_\lambda}(p_1,\dot{q}_1)$, $\dom(r_1)$ $\subseteq\dom(r_0)$ and for each $\xi\in\dom(r_1)$,
$\varrho^\lambda_\xi(p_0,q_0)\Vdash_{\mathbb{A}_{\mathrm{Even}(\xi)}\ast \dot{\mathbb{M}}^\pi_\xi}\dot{r}_0(\xi)\leq \dot{r}_1(\xi). $
\end{defi}
\begin{prop}\label{R*RtruncatedAreIsomorphic2Magidor}
$\mathbb{R}^*$ and $\mathbb{R}\upharpoonright\xi^*$ are isomorphic. In particular, $\mathbb{R}^*$ forces that $\dot{T}$ is a $\lambda$-Aronszajn tree.
\end{prop}
\begin{proof}
It is not hard to check that $(p,\dot{q},r)\mapsto (\varphi(p,\dot{q}),r)$ defines an isomorphism between both forcings.
\end{proof}
We briefly digress from our previous discussion to introduce the notion of $\Pi^1_1$-indescribability,  which will be necessary in future arguments.
\begin{defi}[$\Pi^1_1$-indescribability]
Let $\theta$ be a cardinal and $X\in\mathcal{P}(\theta)$. We will say that $X$ is $\Pi^1_1$-indescribable in $\theta$ if for each $Y\s V_\theta$ and each $\Pi^1_1$ sentence $\Phi$, if
$\langle V_\theta, \in, Y\rangle\models \Phi$, then there is an ordinal $\eta\in X$ be such that $\langle V_\eta,\in, Y\cap V_\eta\rangle\models\Phi$. The cardinal $\theta$ is said to be $\Pi^1_1$-indescribable if $\theta$ is $\Pi^1_1$-indescribable in $\theta$.
\end{defi}
A classical result of Hanf and Scott \cite[Theorem 6.4]{Kan} establishes that $\theta$ is weakly compact if and only if $\theta$ is $\Pi^1_1$-indescribable. Thus, assuming that $\theta$ is weakly compact, it is not hard to prove that
$$\mathcal{F}_\theta:=\{X\in\mathcal{P}(\theta)\mid \text{$``\theta\setminus X$ is not $\Pi^1_1$-indescribable in $\theta$''}\}$$
forms a proper filter on $\theta$. An important property of $\mathcal{F}_\theta$ discovered by Levy is normality \cite[Proposition 6.11]{Kan}. This implies in particular that $\mathcal{F}_\theta$ extends $\mathrm{Club}(\theta)$, \emph{the club filter on $\theta$}, and thus concetrates on the set of Mahlo cardinals below $\theta$. We shall use this to obtain the following  refinement of the set $\mathcal{B}$. To this aim, recall that $H^*$ stands for the generic filter $\pi^*[G\upharpoonright\xi^*]$.
\begin{lemma}\label{DefofB*Magidor}
There is $\mathcal{B}^*\in(\mathcal{F}_\lambda)^V$,  $\mathcal{B}^*\subseteq \mathcal{B}$, with $\kappa^+<\min\mathcal{B}^*$ such that for every $\xi\in\mathcal{B}^*$, $\langle\mathcal{\dot{U}}_\lambda^{\pi^*}(\alpha,\beta)_{H^*}\cap V[H^*\upharpoonright\xi]\mid \alpha\leq \kappa,\,\beta<o^{\mathcal{\dot{U}^{\pi^*}}_{\lambda}}(\alpha)\rangle$ is a coherent sequence of measures in $V[H^*\upharpoonright\xi]$.
\end{lemma}
\begin{proof}
The construction of $\mathcal{B}^*$ is the same as for $\mathcal{B}$ but starting from $\mathcal{B}$ instead of $\lambda$ (c.f. Lemma \ref{ReflectionMeasuresD}). By construction, $\mathcal{B}^*$ is an unbounded set closed by increasing sequences of length $\geq\kappa^+$, hence $\mathcal{B}^*\in (\mathcal{F}_\lambda)^V$.
\end{proof}

\begin{notation}
\rm{
For each $\xi\in\mathcal{B}^*$, let $\mathcal{U}^{\pi^*}_\xi$ denote the sequence witnessing Lemma \ref{DefofB*Magidor}. Set
$\mathbb{M}^{\pi^*}_\xi:=\mathbb{M}_{\mathcal{U}^{\pi^*}_\xi}$.
}
\end{notation}

\begin{lemma}\label{ProjectionsPistaronEven}
Let $\hat{\mathcal{B}^*}=\mathcal{B}^*\cup \{\lambda\}$ and $\xi\leq\eta\in\hat{\mathcal{B}}^*$. There are projections
\begin{enumerate}
\item $\varrho^\eta_\xi:\mathbb{A}_\eta\ast \dot{\mathbb{M}}^{\pi^*}_\eta\rightarrow\mathrm{RO}^+(\mathbb{A}_{\even(\xi)}\ast\dot{\mathbb{M}}^{\pi}_\xi)$,
\item $\tau^\eta_\xi:\mathbb{A}_{\rm{Even}(\eta)}\ast \dot{\mathbb{M}}^{\pi}_\eta\rightarrow\mathrm{RO}^+(\mathbb{A}_{\even(\xi)}\ast\dot{\mathbb{M}}^{\pi}_\xi)$,
\item $\hat{\varrho}^\eta_\xi:\mathrm{RO}^+(\mathbb{A}_\eta\ast \dot{\mathbb{M}}^{\pi^*}_\eta)\rightarrow\mathrm{RO}^+(\mathbb{A}_{\even(\xi)}\ast\dot{\mathbb{M}}^{\pi}_\xi)$.
\item $\hat{\tau}^\eta_\xi:\rm{RO}^+(\mathbb{A}_{\rm{Even}(\eta)}\ast \dot{\mathbb{M}}^{\pi}_\eta)\rightarrow\mathrm{RO}^+(\mathbb{A}_{\even(\xi)}\ast\dot{\mathbb{M}}^{\pi}_\xi)$,
\end{enumerate}
such that $\tau^\lambda_\xi=\hat{\tau}^\eta_\xi\circ \tau^\lambda_\eta$ and $\varrho^\eta_\xi=\hat{\tau}^\eta_\xi \circ \varrho^\xi_\xi$. Moreover, $\varrho^\eta_\xi=\sigma^\eta_\xi$.
\end{lemma}
\begin{proof}
The construction of $\varrho^\eta_\xi$, $\tau^\eta_\xi$, $\hat{\varrho}^\eta_\xi$ and $\hat{\tau}^\eta_\xi$ is analogous to Lemma \ref{ProjectionsCohenPartMagidor}, again using the adequate version of Proposition~\ref{ProjectionGenericsinMagidor}.  A proof for the moreover part can be found in \cite[Lemma 3.18]{FriHon}.
\end{proof}
%The projections $\varrho^\eta_\xi$ will be necessary for the next definition, while the others $\tau^\eta_\xi$ will be used in the definition of the forcing $\mathbb{R}^*_{\rm{Even}}$ and their truncations (cf. Definition \ref{DefREven}).
\begin{defi}[Truncations of $\mathbb{R}^*$]\label{TruncationsR*}
Let $\xi\in\mathcal{B}^*$. A condition in $\mathbb{R}^*\upharpoonright\xi$ is a triple $(p,\dot{q},r)$ for which all the following hold:
\begin{enumerate}
\item $(p,q)\in \mathbb{A}_\xi \ast \dot{\mathbb{M}}^{\pi^*}_\xi$,
\item $r$ is a partial function with $\dom(r)\in[\mathcal{B}^*\cap \xi]^{<\kappa^+}$;
\item For every $\zeta\in\dom(r)$, $r(\zeta)$ is a $\mathbb{A}_{\mathrm{Even}(\zeta)}\ast \dot{\mathbb{M}}^\pi_\zeta$-name such that
$$\one_{\mathbb{A}_{\mathrm{Even}(\zeta)}\ast \dot{\mathbb{M}}^\pi_\zeta}\Vdash_{\mathbb{A}_{\mathrm{Even}(\zeta)}\ast \dot{\mathbb{M}}^\pi_\zeta}\text{$``\dot{r}(\zeta)\in\Add(\kappa^+,1)$''}.$$
\end{enumerate}
For conditions $(p_0,\dot{q}_0,r_0), (p_1,\dot{q}_1,r_1)$ in $\mathbb{R}^*\upharpoonright\xi$ we will write $(p_0,\dot{q}_0,r_0)\leq (p_1,\dot{q}_1,r_1)$ in case $(p_0,\dot{q}_0)\leq_{\mathbb{A}_\xi \ast \dot{\mathbb{M}}^{\pi^*}_\xi} (p_1,\dot{q}_1)$, $\dom(r_1)$ $\subseteq\dom(r_0)$ and for each $\zeta\in\dom(r_1)$,
$\varrho^\xi_\zeta(p_0,q_0)\Vdash_{\mathbb{A}_{\mathrm{Even}(\zeta)}\ast \dot{\mathbb{M}}^\pi_\zeta}\dot{r}_0(\zeta)\leq \dot{r}_1(\zeta). $
\end{defi}
The proof of the next result is analogous to Proposition \ref{ProjectionRandTruncationsMagidor}.
\begin{prop}
For each $\xi\in\mathcal{B}^*$, there is a projection between $\mathbb{R}^*$ and $\mathrm{RO}^+(\mathbb{R}^*\upharpoonright\xi)$. In particular, $\mathbb{R}^*$ is isomorphic to the iteration $\mathbb{R}^*\upharpoonright\xi\ast (\mathbb{R}^*/\mathbb{R}^*\upharpoonright\xi)$.
\end{prop}

\begin{lemma}\label{AronszajnWeakCompMagidor}
Assume there is a $\lambda$-Aronszajn tree $T$ in $V^{\mathbb{R}^*}$. Then there is $\xi\in\mathcal{B}^*$ such that $T\cap \xi$ is a $\xi$-Aronszajn tree in $V^{\mathbb{R}^*\upharpoonright\xi}$.
\end{lemma}
\begin{proof}
Let $\dot{T}$ be a $\mathbb{R}^*$-name such that $\one_{\mathbb{R}^*}\Vdash_{\mathbb{R}^*}\text{``$\dot{T}$ is a $\lambda$-Aronszajn tree''}$. Without loss of generality we may assume that $\dot{T}$ is a $\mathbb{R}^*$-name for a subset of $\lambda$.
It is not hard to check that the above forcing sentence is equivalent to a $\Pi^1_1$ sentence $\Phi$ in the language $\mathcal{L}=\{\in,\mathbb{R}^*,\dot{T},\lambda\}$. % that holds in  $\langle V_\lambda,\in, \mathbb{R}^*,\dot{T},\lambda\rangle$.
Since $\lambda$ is weakly compact, hence $\Pi^1_1$-indescribable, there is a set $X\in(\mathcal{F}_\lambda)^V$ such that for each $\xi\in X$, $\langle V_\xi,\in, \mathbb{R}^*\cap V_\xi, \dot{T}\cap \xi, \xi\rangle\models \Phi$. By Lemma \ref{DefofB*Magidor} and the former discussion we can assume that all these $\xi$ are Mahlo and that $\xi\in\mathcal{B}^*$. In particular, $\mathbb{R}^*\cap V_\xi=\mathbb{R}^*\upharpoonright \xi$, and thus $\langle V_\xi,\in, \mathbb{R}^*\upharpoonright\xi, \dot{T}\cap \xi, \xi\rangle\models \Phi$. Notice that $\Phi$ is absolute between the universe of sets and this structure, hence $\one_{\mathbb{R}^*\upharpoonright\xi}\Vdash_{\mathbb{R}^*\upharpoonright\xi}\text{``$\dot{T}\cap \xi$ is a $\xi$-Aronszajn tree''}$.
%witnessing the reflection of $\Phi$ is a Mahlo cardinal in $\mathcal{B}^*$, hence $\mathbb{R}^*\cap V_\gamma=\mathbb{R}^*\upharpoonright \gamma$,  and thus $\langle V_\gamma,\in, \mathbb{R}^*\upharpoonright\gamma, \dot{T}\cap \gamma, \gamma\rangle\models \Phi$. Since $\Phi$ is absolute between the universe $V$ and the aforementioned structure altogether this proves that $\Vdash_{\mathbb{R}^*\upharpoonright\gamma}\text{``$\dot{T}\cap \gamma$ is a $\gamma$-Aronszajn tree''}$ and thus the lemma follows.
\end{proof}

\begin{lemma}
Assume that there is a $\lambda$-Aronszajn tree $T\s \lambda$ in $V^{\mathbb{R}^*}$. Let $\xi\in\mathcal{B}^*$ be as in the previous lemma. Then $\mathbb{R}^*/(\mathbb{R}^*\upharpoonright\xi)$ adds $b_\xi$, a  cofinal branch throughout $T\cap \xi$.
\end{lemma}
\begin{proof}
Observe that in $V^{\mathbb{R}^*}$ there is a cofinal branch $b_\xi$ for $T\cap\xi$, %throughout $T\cap \gamma$,
 as $T$ is a $\lambda$-tree. Nonetheless, $T\cap \xi$ is $\xi$-Aronszajn in $V^{\mathbb{R}^*\upharpoonright\xi}$ so this branch must be added by the quotient   $\mathbb{R}^*/(\mathbb{R}^*\upharpoonright\xi)$.
\end{proof}
By combining Proposition~\ref{TruncationsAndAronszajnTreeMagidor} and \ref{R*RtruncatedAreIsomorphic2Magidor} with the above lemma  it follows that if the quotients  $\mathbb{R}^*/(\mathbb{R}^*\upharpoonright\xi)$ do not add $\xi$-branches then $\TP(\lambda)$ holds in $V[R]$.
\medskip

In the next series of lemmas we will prove that for each $\xi\in\mathcal{B}^*$ there are forcings $\mathbb{P}_\xi$ and $\mathbb{Q}^{\rm{Even}}_\xi$ fulfilling the following properties:
\begin{itemize}
\item[($\alpha_\xi$)] $\mathbb{P}_\xi\times \mathbb{Q}^{\rm{Even}}_\xi$ projects onto $\mathbb{R}^*/(\mathbb{R}^*\upharpoonright\xi)$ in $V^{\mathbb{R}^*\upharpoonright\xi}$.
\item[($\beta_\xi$)] $\mathbb{P}_\xi\times\mathbb{Q}^{\rm{Even}}_\xi$ does not add new branches to $T\cap \xi$ over $V^{\mathbb{R}^*\upharpoonright\xi}$.
\end{itemize}
Combining ($\alpha_\xi$) and ($\beta_\xi$) we would conclude that $\mathbb{R}^*/(\mathbb{R}^*\upharpoonright\xi)$ does not add $\xi$-branches to $T\cap \xi$. In particular, if this is  true for each $\xi\in\mathcal{B}^*$ then $V[R]\models \TP(\lambda)$.

\smallskip

We now introduce a subforcing of $\mathbb{R}^* \upharpoonright \xi$ which we will use to prove properties $(\alpha_\xi)$ and $(\beta_\xi)$.
%\begin{defi}\label{DefREven}
%A condition in $\mathbb{R}^*_{\rm{Even}}$ is a triple $(p,\dot{q},r)$ for which all the following hold:
%\begin{enumerate}
%\item $(p,q)\in \mathbb{A}_{\rm{Even}(\lambda)} \ast \dot{\mathbb{M}}^{\pi}_\lambda$;
%\item $r$ is a partial function with $\dom(r)\in[\mathcal{B}^*]^{<\kappa^+}$;
%\item  For every $\xi\in\dom(r)$, $r(\xi)$ is an $\mathbb{A}_{\mathrm{Even}(\xi)}\ast \dot{\mathbb{M}}^\pi_\xi$-name such that
%$$\one_{\mathbb{A}_{\mathrm{Even}(\xi)}\ast \dot{\mathbb{M}}^\pi_\xi}\Vdash_{\mathbb{A}_{\mathrm{Even}(\xi)}\ast %\dot{\mathbb{M}}^\pi_\xi}\text{$``\dot{r}(\xi)\in\dot{\Add}(\kappa^+,1)$''}.$$
%\end{enumerate}
%For conditions $(p_0,\dot{q}_0,r_0),(p_1,\dot{q}_1,r_1)$  in $\mathbb{R}^*_{\rm{Even}}$ we will write $(p_0,\dot{q}_0,r_0)\leq %(p_1,\dot{q}_1,r_1)$ in case $(p_0,\dot{q}_0)\leq_{\mathbb{A}_{\mathrm{Even}(\lambda)}\ast \dot{\mathbb{M}}^\pi_\lambda}(p_1,\dot{q}_1)$, %$\dom(r_1)$ $\subseteq\dom(r_0)$ and for each $\xi\in\dom(r_1)$,
%$\tau^\lambda_\xi(p_0,q_0)\Vdash_{\mathbb{A}_{\mathrm{Even}(\xi)}\ast \dot{\mathbb{M}}^\pi_\xi}\dot{r}_0(\xi)\leq \dot{r}_1(\xi). $
%\end{defi}

\begin{defi}\label{TruncationsREven}
Let $\xi\in\mathcal{B}^*$. A condition in the poset $\mathbb{R}^*_{\rm{Even}}\upharpoonright\xi$ is a triple $(p,\dot{q},r)$ for which all the following hold:
\begin{enumerate}
\item $(p,q)\in \mathbb{A}_{\rm{Even}(\xi)}\ast \dot{\mathbb{M}}^{\pi}_\xi$,
\item $r$ is a partial function with $\dom(r)\in[\mathcal{B}^*\cap \xi]^{<\kappa^+}$;
\item For every $\zeta\in\dom(r)$, $r(\zeta)$ is a $\mathbb{A}_{\mathrm{Even}(\zeta)}\ast \dot{\mathbb{M}}^\pi_\zeta$-name such that
$$\one_{\mathbb{A}_{\mathrm{Even}(\zeta)}\ast \dot{\mathbb{M}}^\pi_\zeta}\Vdash_{\mathbb{A}_{\mathrm{Even}(\zeta)}\ast \dot{\mathbb{M}}^\pi_\zeta}\text{$``\dot{r}(\zeta)\in\Add(\kappa^+,1)$''}.$$
\end{enumerate}
For conditions $(p_0,\dot{q}_0,r_0), (p_1,\dot{q}_1,r_1)$ in $\mathbb{R}^*\upharpoonright\xi$ we will write $(p_0,\dot{q}_0,r_0)\leq (p_1,\dot{q}_1,r_1)$ in case $(p_0,\dot{q}_0)\leq_{\mathbb{A}_{\rm{Even}(\xi)} \ast \dot{\mathbb{M}}^{\pi}_\xi} (p_1,\dot{q}_1)$, $\dom(r_1)$ $\subseteq\dom(r_0)$ and for each $\zeta\in\dom(r_1)$,
$\tau^\xi_\zeta(p_0,q_0)\Vdash_{\mathbb{A}_{\mathrm{Even}(\zeta)}\ast \dot{\mathbb{M}}^\pi_\zeta}\dot{r}_0(\zeta)\leq \dot{r}_1(\zeta). $
\end{defi}
Clearly $\mathbb{R}^*\upharpoonright\xi $ projects onto $\mathbb{R}^*_{\rm{Even}}\upharpoonright\xi$, for each $\xi\in\mathcal{B}^*\cup\{\lambda\}$. The following is a key lemma:
%Let $\xi  \in\mathcal{B}^*\cup\{\lambda\}$. For a condition $p\in\mathbb{A}_\xi$, set $p_o:=p\upharpoonright\rm{Odd(\xi)}$ and $p_e:=p\upharpoonright\rm{Even(\xi)}$. Besides, if $\dot{q}$ is a $\mathbb{A}_\xi$-name for a condition in $\dot{\mathbb{M}}^{\pi^*}_\xi$ denote by $\dot{q}_e$ the $\mathbb{A}_{\rm{Even}(\xi)}$-name induced by $\dot{q}$ and the restriction map between $\mathbb{A}_\xi$ and $\mathbb{A}_{\rm{Even}(\xi)}$. Clearly, if $p\Vdash_{\mathbb{A}_\xi}\dot{q}\in\dot{\mathbb{M}}^{\pi^*}_\xi$ then $p_e\Vdash_{\mathbb{A}_{\rm{Even}(\xi)}}\dot{q}_e\in\dot{\mathbb{M}}^{\pi}_\xi$.
\begin{lemma}\label{FactorizationRstarxi}
For each $\xi\in\mathcal{B}^*$, $\psi_\xi:  \mathbb{R}^*\upharpoonright\xi\rightarrow \mathbb{A}_{\rm{Odd}(\xi)} \times \mathbb{R}^*_{\rm{Even}}\upharpoonright\xi$ given by  $ (p,\dot{q},r)\mapsto \langle p\upharpoonright\mathrm{Odd}(\xi), (\varrho^\xi_\xi(p,\dot{q}),r)\rangle$  defines a dense embedding. In particular, both posets are forcing equivalent and thus $V^{\mathbb{R}^*\upharpoonright\xi}$ can be seen as a  $\kappa^+$-cc extension of $V^{\mathbb{R}^*_{\rm{Even}}\upharpoonright\xi}$.
\end{lemma}
\begin{proof}
It is routine to check that $\psi_\xi$ is order-preserving and that it preserves incompatibility. Now let $\langle p', (p, \dot{q},r)\rangle\in \mathbb{A}_{\rm{Odd}(\xi)}\times \mathbb{R}^*_{\rm{Even}}\upharpoonright\xi$. Letting $\dot{q}^*$ the usual identification of $\dot{q}$ as a $\mathbb{A}_\xi$-name it follows that $\psi_\xi((p'\cup p, \dot{q}^*, r))\leq \langle p', (p, \dot{q},r)\rangle.$
\end{proof}

\begin{defi}\label{DefiAfterPropertiesAlphaBetaMagidor}
For each $\xi\in\mathcal{B}^*\cup\{\lambda\}$, define $\mathbb{C}_\xi:=\mathbb{A}_\xi\ast\dot{\mathbb{M}}^{\pi^*}_\xi$, $\mathbb{P}_\xi:=\mathbb{C}_\lambda/\mathbb{C}_\xi$ and $\mathbb{U}_\xi:=\{(\one,\dot{\one},r)\mid (\one,\dot{\one}, r)\in \mathbb{R}^*\upharpoonright\xi\}$. Over $V^{\mathbb{R}^*\upharpoonright\xi}$, define $\mathbb{Q}_\xi:=\{(\one,\dot{\one},r)\mid (\one,\dot{\one},r)\in\mathbb{R}^*/\mathbb{R}^*\upharpoonright\xi\}$.
Also, over $V^{\mathbb{R}_{\rm{Even}}^*\upharpoonright\xi}$, define  $\mathbb{Q}^{\rm{Even}}_\xi:=\{(\one, \dot{\one},r)\mid (\one, \dot{\one},r)\in \mathbb{R}^*/\mathbb{R}^*_{\rm{Even}}\upharpoonright\xi\}$,
\end{defi}
%\begin{defi}
%For each $\xi\in\mathcal{B}^*\cup\{\lambda\}$, define $\mathbb{C}^{\rm{Even}}_\xi:=\mathbb{A}_{\rm{Even}(\xi)}\ast\dot{\mathbb{M}}^{\pi}_\xi$ and %$\mathbb{U}^{\rm{Even}}_\xi:=\{(\one,\dot{\one},r)\mid (\one,\dot{\one}, r)\in \mathbb{R}_{\rm{Even}}^*\upharpoonright\xi\}$. Also, over %$V^{\mathbb{R}_{\rm{Even}}^*\upharpoonright\xi}$, define  $\mathbb{Q}^{\rm{Even}}_\xi:=\{(\one, \dot{\one},r)\mid (\one, \dot{\one},r)\in %\mathbb{R}^*/\mathbb{R}^*_{\rm{Even}}\upharpoonright\xi\}$,
%\end{defi}

Arguing respectively as in Proposition~\ref{projectionUMagidor}  and Proposition \ref{PropertiesOfV[R]Magidor} one obtains the following:
\begin{prop}\label{projectionUgammaMagidor}
For each $\xi\in\mathcal{B}^*$, the following hold:
\begin{enumerate}
\item $\mathbb{U}_\xi$ is $\kappa^+$-directed closed.
\item $\mathbb{C}_\xi\times\mathbb{U}_\xi$ projects onto $\mathbb{R}^*\upharpoonright\xi$ via the map $\langle (p,\dot{q}),(\one,\dot{\one},r)\rangle\mapsto (p,\dot{q},r)$.
\item $V^{\mathbb{C}_\xi}$ and $V^{\mathbb{R}^*\upharpoonright\xi}$ have the same ${<}\kappa^+$-sequences
\end{enumerate}
%Moreover, the same is true for $\mathbb{U}^{\rm{Even}}_\xi$, $\mathbb{C}^{\rm{Even}}_\xi$ and $\mathbb{R}^*_{\rm{Even}}\upharpoonright\xi$.
\end{prop}
\begin{prop}
For each $\xi\in\mathcal{B}^*$, the following hold:
\begin{enumerate}
\item $\mathbb{R}^*\upharpoonright\xi$ is $\xi$-Knaster. In particular, all $V$-cardinals ${\geq}\xi$ are preserved.
\item $\mathbb{R}^*\upharpoonright\xi$ preserves all the cardinals outside the interval $((\kappa^+)^V,\xi)$, while collapses the cardinals there to $(\kappa^+)^V$.  %while it preserves the others.
In particular, $$V^{\mathbb{R}^*\upharpoonright\xi}\models\text{ $``(\kappa^+)^V=\kappa^+\,\wedge\,\xi=\kappa^{++}$''}.$$
\item $V^{\mathbb{R}^*\upharpoonright\xi}\models \text{$``\kappa$ is  strong limit  with $\cof(\kappa)=\delta$''}$.
\item $V^{\mathbb{R}^*\upharpoonright\xi}\models\text{$``2^\kappa\geq \xi$''}$.
\end{enumerate}
%Moreover the same is true for $\mathbb{U}^{\rm{Even}}_\xi$, $\mathbb{C}^{\rm{Even}}_\xi$ and $\mathbb{R}^*_{\rm{Even}}\upharpoonright\xi$.
\end{prop}
\begin{lemma}
For each $\xi\in\mathcal{B}^*$, $\mathbb{Q}^{\rm{Even}}_\xi$ is $\kappa^+$-directed closed over $V^{\mathbb{R}^*_{\rm{Even}}\upharpoonright\xi}$.
\end{lemma}
\begin{proof}
The argument is the same as in the proof of Proposition \ref{projectionUMagidor}(1), using the fact that $V^{\mathbb{R}^*_{\rm{Even}}\upharpoonright\xi}$ and $V^{\mathbb{A}_{\rm{Even}(\xi)}\ast \dot{\mathbb{M}}^{\pi}_\xi}$ have the same subsets of $\kappa.$
\end{proof}
Despite in \cite{FriHon} was claimed that $\mathbb{Q}_\xi$ is $\kappa^+$-closed over $V^{\mathbb{R}^*\upharpoonright\xi}$ this is, unfortunately, not the case: Let $x\in{}^{\kappa}2$ be  in $ V^{\mathbb{R}^*\upharpoonright\xi}\setminus V^{\mathbb{A}_{\rm{Even}(\xi)}\ast \dot{\mathbb{M}}^\pi_\xi}$ and fix $\gamma\in\mathcal{B}^*\cap \xi$. For each $\alpha< \kappa$, set $r_\alpha(\gamma):=x\upharpoonright\alpha\in V$. Clearly, $\langle r_\alpha(\gamma)\mid \alpha<\kappa\rangle$ defines a decreasing sequence of conditions of $\Add(\kappa^+,1)$ in $V^{\mathbb{A}_{\rm{Even}(\gamma)}\ast \dot{\mathbb{M}}^\pi_\gamma}$. Observe however that the only possible value for a lower bound is $r_\kappa(\gamma)=x$, which is not an element of the inner model $V^{\mathbb{A}_{\rm{Even}(\xi)}\ast \dot{\mathbb{M}}^\pi_\xi}$.\footnote{We are very grateful to the anonymous referee for pointing us out this.}
\smallskip

In the next series of lemmas we show  that $\mathbb{P}_\xi\times \mathbb{Q}^{\rm{Even}(\xi)}_\xi$ satisfies $(\alpha_\xi)$ and $(\beta_\xi)$, hence repairing the argument of \cite[\S3]{FriHon}

\begin{lemma}
For each $\xi\in\mathcal{B}^*$, the identity map defines a projection between $\mathbb{Q}^{\rm{Even}}_\xi$ and $\mathbb{Q}_\xi$.
\end{lemma}

\begin{proof}
Clearly $\mathbb{R}^*$ projects onto $\mathbb{R}^*\upharpoonright\xi$ and this latter onto $\mathbb{R}_{\rm{Even}}^*\upharpoonright\xi$. Let us denote each of these projections by $\pi_\xi$ and $\pi^{\rm{Even}}_\xi$, respectively. Observe that it is enough to show that our intended projection is well-defined:  Let $(\one,\dot{\one},r)\in\mathbb{Q}^{\rm{Even}}_\xi$. By definition, $(\one,\dot{\one},\pi^{\rm{Even}}_\xi\circ \pi_\xi(r))\in \dot{G}_{\mathbb{R}^*_{\rm{Even}}\upharpoonright\xi}$. Again, by definition, $(\one,\dot{\one},\pi_\xi(r))\in \dot{G}_{\mathbb{R}^*\upharpoonright\xi}$, which yields $(\one,\dot{\one},r)\in \mathbb{Q}_\xi$, as desired.
\end{proof}

\begin{prop}
For each $\xi\in\mathcal{B}^*$, $\mathbb{P}_\xi\times \mathbb{Q}^{\rm{Even}}_\xi$ satisfies $(\alpha_\xi)$.
\end{prop}
\begin{proof}
By the above lemma it is enough with checking that $\mathbb{P}_\xi\times\mathbb{Q}_\xi$ satisfies $(\alpha_\xi)$.
By definition, a condition in $\mathbb{R}^*/\mathbb{R}^*\upharpoonright\xi$ is a triple $(p,\dot{q},r)$ such that $(\pi^\lambda_\xi(p,\dot{q}),r\upharpoonright\xi)\in \mathbb{R}^*\upharpoonright\xi$, where $\pi^\lambda_\xi$ is the composition of $\varrho^\lambda_\xi$ with the standard isomorphism between $\mathbb{C}_\xi$ and $\mathrm{RO}^+(\mathbb{C}_\xi)$. In particular, $(p,\dot{q})\in \mathbb{P}_\xi$. Now, it is immediate to check that $\tau: \mathbb{P}_\xi\times\mathbb{Q}_\xi\rightarrow\mathbb{R}^*/\mathbb{R}^*\upharpoonright\xi$ given by $\langle(p,\dot{q}),(\one,\one,r)\rangle\mapsto (p,\dot{q},r)$ defines a projection.
\end{proof}
It thus remains to prove that $\mathbb{P}_\xi\times\mathbb{Q}^{\rm{Even}}_\xi$ satisfies $(\beta_\xi)$. For this we will need the following
preservation lemmas from \cite{UngFrag} and \cite{Ung}, respectively.
\begin{lemma}
\label{preservationlemma}\hfill
\begin{enumerate}
  \item [(a)] Assume $2^\tau \geq \eta,$ $\mathbb{P}$ is $\tau^+$-c.c., and $\mathbb{R}$ is $\tau^+$-closed. Suppose $\dot{T}$
  is a $\mathbb{P}$-name for an $\eta$-tree. Then in $V[G_{\mathbb{P}}]$, forcing with $\mathbb{R}$ can not add any new cofinal branches through
  $\dot{T}_{G_{\mathbb{P}}}$.
  \item [(b)] Suppose $\kappa$ is a regular cardinal, $T$ is a $\kappa$-tree and $\mathbb{P}$ is such that $\mathbb{P} \times \mathbb{P}$ is $\kappa$-c.c.
  Then forcing with $\mathbb{P}$ can not add new cofinal branches through $T$.
\end{enumerate}
\end{lemma}
\begin{prop}
Let $\xi\in\mathcal{B}^*$. If $\mathbb{P}_\xi\times\mathbb{P}_\xi$ is $\kappa^+$-cc over $V^{\mathbb{C}_\xi}$ then $\mathbb{P}_\xi\times\mathbb{Q}^{\rm{Even}}_\xi$ witnesses $(\beta_\xi)$.
\end{prop}
\begin{proof}
Let us first prove that if $\mathbb{P}_\xi\times\mathbb{P}_\xi$ is $\kappa^+$-cc over $V^{\mathbb{R}^*\upharpoonright\xi}$ then $\mathbb{P}_\xi\times\mathbb{Q}^{\rm{Even}}_\xi$ witnesses $(\beta_\xi)$. By Proposition \ref{FactorizationRstarxi} we can identify $V^{\mathbb{R}^*\upharpoonright\xi}$ as $V^{\mathbb{A}_{\rm{Odd}(\xi)}\times \mathbb{R}^*_{\rm{Even}} \upharpoonright \xi}$. Clearly, $\mathbb{A}_{\rm{Odd}(\xi)}\ast (\mathbb{P}_\xi\times\mathbb{P}_\xi)$ is $\kappa^+$-cc over $V^{\mathbb{R}^*_{\rm{Even}} \upharpoonright \xi}$.

Let $G_{\mathbb{R}^*_{\rm{Even}} \upharpoonright \xi}$ be a $\mathbb{R}^*_{\rm{Even}} \upharpoonright \xi$-generic filter over $V$
and
let $\tau \in V[G_{\mathbb{R}^*_{\rm{Even}} \upharpoonright \xi}]$ be an $\mathbb{A}_{\rm{Odd}(\xi)} (\cong \mathbb{R}^*\upharpoonright\xi / \dot{G}_{\mathbb{R}^*_{\rm{Even}} \upharpoonright \xi})$-name for $T \cap \xi.$ Then we can consider $\tau$ as an $\mathbb{A}_{\rm{Odd}(\xi)} \ast \mathbb{P}_\xi$-name for $T \cap \xi$ as well.
It then follows from Lemma \ref{preservationlemma}(a) that the tree $T \cap \xi$ has the same cofinal branches in the models
\[
V[G_{\mathbb{R}^*_{\rm{Even}} \upharpoonright \xi}][G_{\mathbb{A}_{\rm{Odd}(\xi)}} \ast G_{\mathbb{P}_\xi}][G_{\mathbb{Q}^{\rm{Even}}_\xi}]
\]
and
\[
V[G_{\mathbb{R}^*_{\rm{Even}} \upharpoonright \xi}][G_{\mathbb{A}_{\rm{Odd}(\xi)}} \ast G_{\mathbb{P}_\xi}]
\]
On the other hand,  recall that $T \cap \xi$ had no cofinal branches in $V^{\mathbb{R}^* \upharpoonright \xi}=V^{\mathbb{R}^*_{\rm{Even}} \upharpoonright \xi \ast \mathbb{A}_{\rm{Odd}(\xi)}}$. By our assumption, $\mathbb{P}_\xi\times\mathbb{P}_\xi$ is $\kappa^+$-cc over $V^{\mathbb{R}^*\upharpoonright\xi}$ hence, by Lemma \ref{preservationlemma}(b), $T \cap \xi$ has no cofinal branches
in $V[G_{\mathbb{R}^*_{\rm{Even}} \upharpoonright \xi}][G_{\mathbb{A}_{\rm{Odd}(\xi)}} \ast G_{\mathbb{P}_\xi}].$
The result follows as
$$V[G_{\mathbb{R}^*_{\rm{Even}} \upharpoonright \xi}][G_{\mathbb{A}_{\rm{Odd}(\xi)}} \ast G_{\mathbb{P}_\xi}][G_{\mathbb{Q}^{\rm{Even}}_\xi}]=V[G_{\mathbb{R}^* \upharpoonright \xi}][ G_{\mathbb{P}_\xi} \times G_{\mathbb{Q}^{\rm{Even}}_\xi}]
.$$
We are now left with showing that if $\mathbb{P}_\xi\times\mathbb{P}_\xi$ is $\kappa^+$-cc over $V^{\mathbb{C}_\xi}$ then it is also $\kappa^+$-cc over $V^{\mathbb{R}^*\upharpoonright\xi}$. Indeed, observe that then $\mathbb{C}_\xi \ast (\mathbb{P}_\xi\times \mathbb{P}_\xi)$ is $\kappa^+$-cc over $V$, hence, by Easton's lemma, this is $\kappa^+$-cc over $V^{\mathbb{U}_\xi}$. Thus, $\mathbb{P}_\xi\times\mathbb{P}_\xi$ is $\kappa^+$-cc over $V^{\mathbb{C}_\xi\times \mathbb{U}_\xi}$. Since $\mathbb{C}_\xi\times\mathbb{U}_\xi$ projects onto $\mathbb{R}^*\upharpoonright\xi$ the desired result follows.
\end{proof}
Therefore, we are left with showing that $\mathbb{P}_{\xi}\times\mathbb{P}_\xi$ is $\kappa^+$-cc over $V^{\mathbb{C}_\xi}$. We will devote the next series of lemmas to this purpose.

\begin{lemma}\label{PreviousNotInTheQuotient}
Let $\mathbb{P}$ and $\mathbb{Q}$ be  two forcing notions and $\pi:\mathbb{P}\rightarrow\mathbb{Q}$ be a projection. For every $p\in\mathbb{P}$ and $q\in\mathbb{Q}$, $q\Vdash_\mathbb{Q} p\notin \dot{(\mathbb{P}/\mathbb{Q})}$ if and only if for every generic filter $G\subseteq\mathbb{P}$ with $p\in G$, $q$ is not in $H$, the generic filter generated by $\pi[ G]$. In particular, $\pi(p)\perp q$ iff $q\Vdash_{\mathbb{Q}} p\notin \dot{(\mathbb{P}/\mathbb{Q})}$.
\end{lemma}
\begin{proof}
The first implication is obvious. Conversely, assume that there is $q'\leq_\mathbb{Q} q$ be such that $q'\Vdash_\mathbb{Q} p\in \dot{(\mathbb{P}/\mathbb{Q})}$. Let $H\s\mathbb{Q}$ be some generic filter over $V$ containing $q$. Hence, $p\in\mathbb{P}/H$. Now let $G\s\mathbb{P}/H$ be some generic filter over $V[H]$ containing $p$. Clearly $\pi[G]=H$ and $q\in H$, which yields the desired contradiction.
\end{proof}
\begin{convention}
\rm{
Let $\mathbb{M}$ be Magidor forcing with respect to a coherent sequence of measures $\mathcal{V}$. Hereafter, we will identify each condition $p\in\mathbb{M}$ with the sequence $\langle\vec{\alpha}^p,\vec{A}^p\rangle$, where $\vec{\alpha}^p:=\langle\alpha^p_0,\dots,\alpha^p_{n^p}\rangle$ and $\vec{A}^p:=\langle A^p_0,\dots, A^p_{n^p}\rangle$. We will tend to  omit the superscript,  which will mean that $\langle\vec{\alpha},\vec{A}\rangle=\langle\vec{\alpha}^p,\vec{A}^p\rangle$, for some $p$ in the corresponding Magidor forcing.
}
\end{convention}

\begin{remark}
\rm{
Let $\xi\in\mathcal{B}^*\cup\{\lambda\}$. Observe that $\tilde{\mathbb{C}}_\xi$, the subposet of $\mathbb{C}_\xi$ with  conditions $(p,\dot{q})$ such that
$p\Vdash_{\mathbb{A}_\xi} \dot{q}=\langle \check{\vec{\alpha}}^q,\dot{\vec{A}}^q\rangle,$
is dense  in $\mathbb{C}_\xi$. Thus, for our current purposes it is enough to assume that $\mathbb{C}_\xi=\tilde{\mathbb{C}}_\xi$.
}
\end{remark}

\begin{lemma}\label{NotNotInTheQuotientMagidor}
Let $\xi\in\mathcal{B}^*$, $r=(p,\langle\check{\vec{\alpha}},\dot{\vec{A}}\rangle)\in\mathbb{C}_\lambda$ and $r'=(q, \langle\check{\vec{\beta}},\dot{\vec{B}}\rangle)\in\mathbb{C}_\xi$, Then, $r'\Vdash_{\mathbb{C}_\xi} \text{$``r\notin\mathbb{P}_\xi$''}$ if and only if one of the following hold:
\begin{enumerate}
\item $p\upharpoonright\xi\perp_{\mathbb{A}_\xi} q$;
%\item $p\upharpoonright\xi\parallel_{\mathbb{A}_\xi} q$ and $\vec{\alpha}\cup \vec{\beta}$ is not an increasing sequence;
\item $p\upharpoonright\xi\parallel_{\mathbb{A}_\xi} q$ and %$\vec{\alpha}\cup \vec{\beta}$% is an increasing function and
$$p\cup q\Vdash_{\mathbb{A}_\lambda} \langle\check{\vec{\beta}},\dot{\vec{B}}\rangle{}^\curvearrowright (\check{\vec{\alpha}}\setminus \check{\vec{\beta}})\notin\dot{\mathbb{M}}^{\pi^*}_\xi\,\vee\, \langle\check{\vec{\alpha}},\dot{\vec{A}}\rangle{}^\curvearrowright (\check{\vec{\beta}}\setminus \check{\vec{\alpha}})\notin\dot{\mathbb{M}}^{\pi^*}_\lambda.\footnote{Here we are identifying the $\mathbb{A}_\xi$-name $\mathbb{M}^{\pi^*}_\xi$ with its standard extension to a $\mathbb{A}_\lambda$-name.}$$
\end{enumerate}
\end{lemma}
\begin{proof}
First, observe that two conditions $\langle\vec{\alpha},\vec{A}\rangle, \langle\vec{\beta},\vec{B}\rangle\in \mathbb{M}^{\pi^*}_\lambda$ are compatible if and only if %$\vec{\alpha}\cup \vec{\beta}$ is increasing and
$\langle\vec{\alpha},\vec{A}\rangle{}^\curvearrowright (\vec{\beta}\setminus \vec{\alpha}), \langle\vec{\beta},\vec{B}\rangle{}^\curvearrowright (\vec{\alpha}\setminus \vec{\beta})\in\mathbb{M}^{\pi^*}_\lambda$. Thereby, if some of the above conditions is true, $\varrho^\lambda_\xi(r)\perp_{\mathbb{C}_\xi} r'$. Thus, Lemma \ref{PreviousNotInTheQuotient} yields $r'\Vdash_{\mathbb{C}_\xi} \text{$``r\notin\mathbb{P}_\xi$''}$. Conversely, assume that (1) and (2) are false and set $\vec{\gamma}:=\vec{\alpha}\cup \vec{\beta}$. Since (1) is false,  $p\cup q\in \mathbb{A}_\lambda$. Also, since (2) is false, we may let  a condition $a\leq_{\mathbb{A}_\lambda} p\cup q$ forcing the opposite. Let $A\s\mathbb{A}_\lambda$ generic (over $V$) containing $a$. By the above, in $V[A]$,  $\langle \vec{\beta},\vec{B}\rangle{}^\curvearrowright(\vec{\alpha}\setminus \vec{\beta})\in \mathbb{M}^{\pi^*}_\xi$ and $\langle\vec{\alpha},\vec{A}\rangle{}^\curvearrowright(\vec{\beta}\setminus \vec{\alpha})\in \mathbb{M}^{\pi^*}_\lambda$, hence both Magidor conditions are compatible. Let $\langle \vec{\gamma},\vec{C}\rangle\in\mathbb{M}^{\pi^*}_\lambda$  be a condition witnessing this compatibility and $S\s\mathbb{M}^{\pi^*}_\lambda$ be generic (over $V[A]$) containing $\langle \vec{\gamma},\vec{C}\rangle$. Set $r^*:=(a, \langle \check{\vec{\gamma}},\dot{\vec{C}}\rangle)$. Clearly, $r^*\in A\ast\dot{S}$ and $r^*\leq_{\mathbb{C}_\lambda} r$, so $r\in A\ast\dot{S}$. On the other hand, $\varrho^\lambda_\xi[A\ast\dot{S}]$ generates a $\mathbb{C}_\xi$-generic filter containing $r'$, hence Lemma \ref{PreviousNotInTheQuotient} yields  $r'\nVdash_{\mathbb{C}_\xi} \text{$``r\notin\mathbb{P}_\xi$''}$, as wanted.
\end{proof}

For each $\xi\in\mathcal{B}^*\cup\{\lambda\}$ we will assume that for each $r=(p, \langle\check{\vec{\alpha}},\dot{\vec{A}}\rangle)\in \mathbb{C}_\xi$, $p\Vdash_{\mathbb{A}_\xi}\text{$``\langle\check{\vec{\alpha}},\dot{\vec{A}}\rangle$ is pruned''}$. This is of course feasible by virtue of  Proposition \ref{ThereisPrunedInMagidor}.
\begin{lemma}\label{RefinningWithRowbottomMagidor}
Let $\xi\in\mathcal{B}^*$, $r=(p,\langle\check{\vec{\alpha}},\dot{\vec{A}}\rangle)\in\mathbb{C}_\lambda$ and $r'=(q, \langle\check{\vec{\beta}},\dot{\vec{B}}\rangle)\in\mathbb{C}_\xi$. Assume that $q\leq_{\mathbb{A}_\xi} p\upharpoonright\xi$,  $\vec{\alpha}\s \vec{\beta}$ and $$(\Upsilon)\;\; p\cup q\Vdash_{\mathbb{A}_\lambda}\text{$`` \forall\gamma\in\vec{\beta}\setminus \vec{\alpha}
\, \left(\dot{\vec{A}}(\check{k}_\gamma)\cap  \gamma\in\dot{\mathcal{F}}(\gamma)\right)$''},$$
where $k_\gamma:=\min \{k< |\vec{\alpha}|\mid \gamma< \vec{\alpha}(k)\}$.
Then there is a $\mathbb{A}_\xi$-name $\dot{\vec{C}}$  for which all the following hold:
\begin{enumerate}
\item[(I)]  $q\Vdash_{\mathbb{A}_\xi}\text{$``q^*:=\langle\check{\vec{\beta}},\dot{\vec{C}}\rangle\leq_{\mathbb{M}^{\pi^*}_\xi} \langle\check{\vec{\beta}},\dot{\vec{B}}\rangle\, \wedge\, q^*$ is pruned''}$.
\item[(II)] $q\Vdash_{\mathbb{A}_\xi}\text{$``\forall \tau\in[\biguplus_i \dot{\vec{C}}(i)]^{<\omega} \left( p\nVdash_{\mathbb{A}_\lambda/\mathbb{A}_\xi} \langle\check{\vec{\alpha}},\dot{\vec{A}}\rangle{}^\curvearrowright \tau\notin\dot{\mathbb{M}}^{\pi^*}_\lambda \right)$''}$.
\end{enumerate}
\end{lemma}
\begin{proof}
Let us work over $V^{\mathbb{A}_\xi\downarrow q}$.  Let $c:[\biguplus_{i\leq|\vec{\beta}|} \vec{B}(i)]^{<\omega}\rightarrow 2$ be defined as
$$
c(\vec{x}):=\begin{cases}
0, & \text{if $p\Vdash_{\mathbb{A}_\lambda/\mathbb{A}_\xi} \langle\check{\vec{\alpha}},\dot{\vec{A}}\rangle{}^\curvearrowright \vec{x}\notin \dot{\mathbb{M}}^{\pi^*}_\lambda$};\\
1, & \text{if $p\nVdash_{\mathbb{A}_\lambda/\mathbb{A}_\xi} \langle\check{\vec{\alpha}},\dot{\vec{A}}\rangle{}^\curvearrowright \vec{x}\notin \dot{\mathbb{M}}^{\pi^*}_\lambda$}.
\end{cases}
$$
There is  %$\vec{C}\s \vec{B}$ homogeneous for $c$.
a sequence $\vec{C}:=\langle C_i\mid i<|\vec{\beta}|\rangle$, such that $C_i\subseteq \vec{B}(i)$ is in $\mathcal{F}(\beta_i)$, $C_i=\biguplus_\alpha C_i(\alpha)$ with $C_i\in\mathcal{U}(\beta_i,\alpha)$, and $\vec{C}$ is homogeneous in the sense of Lemma \ref{RowbottomMagidor}.
In particular, $\langle {\vec{\beta}},{\vec{C}}\rangle\leq_{\mathbb{M}^{\pi^*}_\xi} \langle {\vec{\beta}},{\vec{B}}\rangle$.  By refining $\vec{C}$ we may further assume that $\langle {\vec{\beta}},{\vec{C}}\rangle$ is pruned (cf.
Proposition \ref{ThereisPrunedInMagidor}).
  Thus, (I)  holds. Towards a contradiction, assume that (II) is false. Let $r\leq_{\mathbb{A}_\xi} q$ be such that $r$ forces the negation of the above formula. By shrinking $r$ we may assume that there is a block sequence $\vec{x}$ for $q^*$ such that $r\Vdash_{\mathbb{A}_\xi}\text{$``\left( p\Vdash_{\mathbb{A}_\lambda/\mathbb{A}_\xi} \langle\check{\vec{\alpha}},\dot{\vec{A}}\rangle{}^\curvearrowright \check{\vec{x}}\notin \dot{\mathbb{M}}^{\pi^*}_\lambda\right) $''} $. Since $r\leq_{\mathbb{A}_\xi} q$, $r\cup p\in\mathbb{A}_\lambda$, hence $r\cup p\Vdash_{\mathbb{A}_\lambda} \langle\check{\vec{\alpha}},\dot{\vec{A}}\rangle{}^\curvearrowright \check{\vec{x}}\notin \dot{\mathbb{M}}^{\pi^*}_\lambda$. Now, since $r$ forces $\dot{\vec{C}}$ to be homogeneous for $\dot{c}$ (in the sense of Lemma \ref{RowbottomMagidor}),  it follows that for all block sequence $\vec{y}$ with the same length as $\vec{x}$,  %and with $i_{\vec{y}}=i_{\vec{x}}$,
  if for all $k<|\vec{x}|$, $\vec{x}(k)$ and $\vec{y}(k)$ belong to the same set  $C_i(\alpha)$, then $r\cup p\Vdash_{\mathbb{A}_\lambda} \langle\check{\vec{\alpha}},\dot{\vec{A}}\rangle{}^\curvearrowright \check{\vec{y}}\notin\dot{\mathbb{M}}^{\pi^*}_\lambda$. Let $a$ consists of pairs $(i,\alpha)$ where for some $k<|\vec{x}|$, $\vec{x}(k)\in C_i(\alpha)$. %$a\in\mathcal{P}(|\vec{\beta}|+1)$ be this last common value.
  Since $p$ forces $\langle\check{\vec{\alpha}},\dot{\vec{A}}\rangle$ to be pruned  the only chance for this property to hold is that  $r\cup p \Vdash_{\mathbb{A}_\lambda} \text{$``\exists (i,\alpha)\in a,~\dot{C}_i(\alpha)\cap \dot{\vec{A}}(\check{k}_{\beta_i})(\alpha)\cap \beta_i=\emptyset$''}$, where
 %$r\cup p \Vdash_{\mathbb{A}_\lambda} \text{$`` \exists i\in a\,\, (\dot{\vec{C}}(i)\cap \dot{\vec{A}}(\check{k}_{\beta_i})\cap \beta_i=\emptyset)$''}$,
 $k_{\beta_i}:=\min\{k< |\vec{\alpha}|\mid \beta_i\leq \vec{\alpha}(k)\}$. Now let $A\s\mathbb{A}_\lambda$ be a generic filter with $r\cup p\in A$. Then the above proper\-ty would hold at $V[A]$, which implies that there is some $(i, \alpha) \in a$ be such that $\dot{{C}}_i(\alpha)_G\cap \dot{\vec{A}}(k_{\beta_i})(\alpha)_G\cap \beta_i=\emptyset$. By $(\Upsilon)$,  $\dot{\vec{A}}(k_{\beta_i})_G\cap \beta_i\in \mathcal{F}(\beta_i)$.   It thus follows that
  $\dot{{C}}_i(\alpha)_G $
  and $\dot{\vec{A}}(k_{\beta_i})(\alpha)_G\cap \beta_i$ are in $\mathcal{U}(\beta_i, \alpha)$, which yields the desired contradiction.
\end{proof}

\begin{lemma}\label{LyingWithinTheQuotient}
Let $\xi\in\mathcal{B}^*$, $r=(p,\langle\check{\vec{\alpha}},\dot{\vec{A}}\rangle)\in\mathbb{C}_\lambda$ and $r'=(q, \langle\check{\vec{\beta}},\dot{\vec{B}}\rangle)\in\mathbb{C}_\xi$. Assume that
\begin{enumerate}
\item[$(\aleph)$] $q\leq_{\mathbb{A}_\xi} p\upharpoonright\xi$;
\item[$(\beth)$] $\vec{\alpha}\s \vec{\beta}$;
%\item[$(\gimel)$] $p\cup q\Vdash_{\mathbb{A}_\lambda}\text{$`` \forall\theta\in\dom(\dot{H})\, \left(\dot{H}(\theta)\cap \dot{\mathcal{P}}_{\kappa_{\check{f}(\tau_\theta)}}(\kappa_\theta\cap \check{f}(\tau_\theta))\in \dot{U}^\theta_{\tau_\theta, \check{f}(\tau_\theta)}\right)$''}.$;
\item[$(\gimel)$] $p\cup q\Vdash_{\mathbb{A}_\lambda}\text{$`` \langle\check{\vec{\alpha}},\dot{\vec{A}}\rangle{}^\curvearrowright (\check{\vec{\beta}}\setminus \check{\vec{\alpha}})\in\dot{\mathbb{M}}^{\pi^*}_\lambda$''}.$
\end{enumerate}
Let $\dot{\vec{C}}$ be the sequence obtained from Lemma \ref{RefinningWithRowbottomMagidor} with respect to $r$ and $r'$. Then, $(q,\langle\check{\vec{\beta}},\dot{\vec{C}}\rangle)\Vdash_{\mathbb{C}_\xi} (p,\langle\check{\vec{\alpha}},\dot{\vec{A}}\rangle)\in \mathbb{P}_\xi$.
\end{lemma}
\begin{proof}
Otherwise, let $r^*:=(r,\langle\check{\vec{\gamma}},\dot{\vec{D}}\rangle)\leq_{\mathbb{C}_\xi}(q,\langle\check{\vec{\beta}},\dot{\vec{C}}\rangle)$ be forcing the opposite. By using Lemma \ref{NotNotInTheQuotientMagidor} with respect to $r^*$ and $r$ it follows that either (1) or (2) must hold. It is not hard to check that $(\aleph)$-$(\gimel)$ implies that (2) holds: particularly, that $r\cup p\Vdash_{\mathbb{A}_\lambda}\text{$`` \langle\check{\vec{\alpha}},\dot{\vec{A}}\rangle{}^\curvearrowright (\check{\vec{\gamma}}\setminus \check{\vec{\alpha}})\notin\dot{\mathbb{M}}^{\pi^*}_\lambda$''}$ holds. By $(\gimel)$ and since $r\cup p\leq_{\mathbb{A}_\lambda} p\cup q$, $r\cup p\Vdash_{\mathbb{A}_\lambda}\text{$`` \langle\check{\vec{\alpha}},\dot{\vec{A}}\rangle{}^\curvearrowright (\check{\vec{\gamma}}\setminus \check{\vec{\beta}})\notin\dot{\mathbb{M}}^{\pi^*}_\lambda$''}.$ Clearly, $r\leq_{\mathbb{A}_\xi} q$ and $r\Vdash_{\mathbb{A}_\xi} \check{\vec{\gamma}}\setminus \check{\vec{\beta}}\in[\biguplus_i\dot{\vec{C}}(i)]^{<\omega}$. Observe that $(\gimel)$ yields $(\Upsilon)$ of  Lemma \ref{RefinningWithRowbottomMagidor}, and this latter implies $r\cup p\nVdash_{\mathbb{A}_\lambda}\text{$`` \langle\check{\vec{\alpha}},\dot{\vec{A}}\rangle{}^\curvearrowright (\check{\vec{\gamma}}\setminus \check{\vec{\beta}})\notin\dot{\mathbb{M}}^{\pi^*}_\lambda$''}$, which produces the desired contradiction.
\end{proof}

\begin{lemma}\label{LemmaCCnessMagidor}
Let $\xi\in\mathcal{B}^*$, $(q, \langle\check{\vec{\beta}},\dot{\vec{B}}\rangle)\in\mathbb{C}_\xi$ and $\dot{r}_0, \dot{r}_1$ be two $\mathbb{C}_\xi$-names forced by $\one_{\mathbb{C}_\xi}$ to be  in $\mathbb{P}_\xi$. Then, there are $(q^*, \langle\check{\vec{\beta}}^*,\dot{\vec{B}}^*\rangle)\in\mathbb{C}_\xi$,  $(p_0,\langle\check{\vec{\alpha}}_0,\dot{\vec{A}}_0\rangle)$, $(p_1,\langle\check{\vec{\alpha}}_1,\dot{\vec{A}}_1\rangle)\in \mathbb{P}_\xi$ and $\bar{p}_0, \bar{p}_1\in \mathbb{A}_\lambda$ be such that the following hold: For  $i\in\{0,1\}$,
\begin{itemize}
\item[(a)] $(q^*, \langle\check{\vec{\beta}}^*,\dot{\vec{B}}^*\rangle)\leq_{\mathbb{C}_\xi}(q, \langle\check{\vec{\beta}},\dot{\vec{B}}\rangle)$,
\item[$(b_i)$] $(q^*, \langle\check{\vec{\beta}}^*,\dot{\vec{B}}^*\rangle)\Vdash_{\mathbb{C}_\xi}\text{$``\dot{r}_i=(p_i,\langle\check{\vec{\alpha}}_i,\dot{\vec{A}}_i\rangle)\in\mathbb{P}_\xi$''}$,
\item[$(c_i)$] $\bar{p}_i\leq_{\mathbb{A}_\lambda} p_i$ and $(q^*, \langle\check{\vec{\beta}}^*,\dot{\vec{B}}^*\rangle)$ and $(\bar{p}_i,\langle\check{\vec{\alpha}}_i,\dot{\vec{A}}_i\rangle)$ satisfy conditions (1)-(3) of Lemma \ref{LyingWithinTheQuotient}.
\end{itemize}
\end{lemma}
\begin{proof}
Let $(q^*, \langle\check{\vec{\beta}}^*,\dot{\vec{B}}^*\rangle)\leq_{\mathbb{C}_\xi}(q, \langle\check{\vec{\beta}},\dot{\vec{B}}\rangle)$ and $(p_0,\langle\check{\vec{\alpha}}_0,\dot{\vec{A}}_0\rangle)$, $(p_1,\langle\check{\vec{\alpha}}_1,\dot{\vec{A}}_1\rangle)\in \mathbb{P}_\xi$ be such that $(b_0)$ and $(b_1)$ hold. By extending $q^*$ and $\vec{\beta}^*$ if necessary, we may further assume that $q^*\leq_{\mathbb{A}_\xi} p_0\upharpoonright\xi\cup p_1\upharpoonright\xi$ and $\vec{\alpha}_0\cup \vec{\alpha}_1\s \vec{\beta}^*$. For each $i\in\{0,1\}$, combining this with Lemma \ref{NotNotInTheQuotientMagidor} it follows that condition (3) must fail. Thus, there is $\bar{p}_i\leq_{\mathbb{A}_\lambda} q^*\cup p_i$ with  $\bar{p}_i\Vdash_{\mathbb{A}_\lambda} \langle\check{\vec{\alpha}}_i,\dot{\vec{A}}_i\rangle{}^\curvearrowright (\check{\vec{\beta}}^*\setminus\check{\vec{\alpha}}_i)\in \dot{\mathbb{M}}^{\pi^*}_\lambda$. Again, extend $p^*$ to ensure $q^*\leq_{\mathbb{A}_\xi} \bar{p}_0,\bar{p}_1$. It should be  clear at this point that, for $i\in\{0,1\}$, $(q^*, \langle\check{\vec{\beta}}^*,\dot{\vec{B}}^*\rangle)$ and $(\bar{p}_i,\langle\check{\vec{\alpha}}_i,\dot{\vec{A}}_i\rangle)$ witness $(c_i)$.
\end{proof}
We are finally in conditions to prove the $\kappa^+$-ccness of  $\mathbb{P}_\xi\times\mathbb{P}_\xi$. The proof is exactly the same as in \cite[Lemma 3.28]{FriHon} so we refer the reader there for details.
\begin{lemma}
Let $\xi\in\mathcal{B}^*$. Then, $\one_{\mathbb{C}_\xi}\Vdash_{\mathbb{C}_\xi} \text{$``\mathbb{P}_\xi\times\mathbb{P}_\xi$ is $\kappa^+$-cc''}$.
\end{lemma}
\section{Forcing arbitrary failures of the $\sch_\kappa$}\label{SectionGapArbitrary}
In this last section we address the issue of obtaining arbitrary failures for the $\sch_\kappa$ in the generic extension of Section \ref{TPkappa++SectionMagidor}. Rather than providing full details we will just enumerate the necessary modifications in the arguments. After our exposition we hope to have convinced the reader that the results proved through Section \ref{TPkappa++SectionMagidor} still apply in the current context.
\begin{enumerate}
\item Assume the $\gch_{\geq \kappa}$. Let $\kappa<\lambda$ be a strong and a weakly compact cardinal, respectively. Let $\Theta\geq \lambda^{++}$ be a cardinal with $\cof(\Theta)>\kappa$ and $\delta=\cof(\delta)<\kappa$. By preparing the ground model we may further assume that $\kappa$ is a strong cardinal which is indestructible under adding Cohen subsets of $\kappa$. This preparatory forcing does not mess up our initial hypotheses.

\item Set $\mathbb{A}_\Theta:=\Add(\kappa,\Theta)$ and, for each $x\in [\Theta]^\lambda$,  $\mathbb{A}_x:=\Add(\kappa,x)$. Let $G\s \mathbb{A}_\Theta$ a generic filter over $V$ and $\mathcal{U}\in V[G]$ be a coherent sequence of measures with $\ell^{\mathcal{U}}=\kappa+1$ and $o^{\mathcal{U}}(\kappa)=\delta$. For each pair $(\alpha,\beta)\in\dom(\mathcal{U})$, let $\dot{\mathcal{U}}(\alpha,\beta)$ be a $\mathbb{A}_{\Theta}$-name such that $\mathcal{U}(\alpha,\beta)=\dot{\mathcal{U}}(\alpha,\beta)_G$. Arguing as in Lemma \ref{LemmaD+} we can prove the following:

\begin{lemma}
There exists an  unbounded set $\mathcal{A}\subseteq [\Theta]^\lambda$, closed under taking limits of ${\geq}\kappa^+$-se\-quences, such that, for every  $x\in \mathcal{A}$ and every $\mathbb{A}_{\Theta}$-generic filter $\bar{G}$,
$\mathcal{U}_{x}:=\langle\dot{\mathcal{U}}(\alpha,\beta)_{\bar{G}}\cap V[\bar{G}\upharpoonright x]\mid \alpha\leq\kappa,\, \beta<o^{\dot{\mathcal{U}}}(\alpha)\rangle$
 is a coherent sequence of measures in $V[\bar{G}\upharpoonright x]$.
\end{lemma}
Here we are taking advantage of Notation \ref{CohenForcingNotationMagidor}. Let $\dot{\mathcal{U}}_{x}$ be a $\mathbb{A}_x$-name such that $\mathcal{U}_x=(\dot{\mathcal{U}}_{x})_{G\upharpoonright x}$ Similarly, let $\dot{\mathbb{M}}_x$ denote a $\mathbb{A}_x$-name such that $\mathbb{M}_{\mathcal{U}_x}=(\dot{\mathbb{M}}_{\mathcal{U}_\xi})_{G\upharpoonright x}$. By convention, $\mathcal{U}_{\Theta}:=\mathcal{U}$ and $\dot{\mathbb{M}}_{\Theta}$ will denote a $\mathbb{A}_{\Theta}$-name such that $\mathbb{M}_{\mathcal{U}_\Theta}=(\dot{\mathbb{M}}_{\Theta})_G$. Arguing as in Proposition \ref{ProjectionGenericsinMagidor} one may argue that $\mathbb{M}_{\Theta}$ projects onto $\mathbb{M}_x$, for each $x\in\mathcal{A}$.

\item Choose $x_0\in\mathcal{A}$ with $\lambda+1\s x$ be arbitrary and let $\pi: \mathbb{A}_{x_0}\rightarrow \mathbb{A}_{ \mathrm{Even}(\lambda)}$ be an isomorphism.
Define $\dot{\mathcal{U}}^\pi_{x_0}:= \pi(\dot{\mathcal{U}}_{x_0})$. Clearly, $(\dot{\mathcal{U}}^\pi_{x_0})_{\pi[G\upharpoonright x_0]}=(\dot{\mathcal{U}}_{x_0})_{G\upharpoonright x_0}=\mathcal{U}_{x_0}$.  Say that $\mathcal{U}^\pi_{x_0}(\alpha,\beta)$
are the measures of $\mathcal{U}^\pi_{x_0}$ and that $\dot{\mathcal{U}}^\pi_{x_0}(\alpha,\beta)$ is a $\mathbb{A}_{x_0}$-name such that $\mathcal{U}^\pi_{x_0}(\alpha,\beta)=\dot{\mathcal{U}}^\pi_{x_0}(\alpha,\beta)_{\pi[G\upharpoonright x_0]}$. Let $\mathcal{B}\s\lambda$ be as given in Lemma \ref{ReflectionMeasuresD}. From this point on we will be relying on Notation \ref{NotationUpi}.

\item Set $\hat{\mathcal{A}}:=\{x\in\mathcal{A}\mid x_0\subseteq x\}$. Arguing as in Lemma \ref{ProjectionsCohenPartMagidor} we obtain a system of projections
$$
\langle \sigma^{\Theta}_x: \mathbb{A}_{\Theta}\ast \dot{\mathbb{M}}_\Theta\rightarrow \mathrm{RO}^+(\mathbb{A}_x\ast \dot{\mathbb{M}}_x)\mid \text{ $x\in\hat{\mathcal{A}}$}\rangle,
$$
$$
\langle \hat{\sigma}^{x}_\xi: \mathrm{RO}^+(\mathbb{A}_x\ast \dot{\mathbb{M}}_x)\rightarrow \mathrm{RO}^+(\mathbb{A}_{\mathrm{Even}(\xi)}\ast \dot{\mathbb{M}}^\pi_\xi)\mid \text{ $x\in\hat{\mathcal{A}}$, $\xi\in\mathcal{B}$}\rangle,
$$
$$
\langle {\sigma}^\Theta_\xi: \mathbb{A}_{\Theta}\ast \dot{\mathbb{M}}_\Theta\rightarrow \mathrm{RO}^+(\mathbb{A}_{\mathrm{Even}(\xi)}\ast \dot{\mathbb{M}}^\pi_\xi)\mid \text{$\xi\in\mathcal{B}$}\rangle.
$$
Also, one may guarantee that these projections commute; namely,
$
\sigma^{\Theta}_\xi=\hat{\sigma}^{x}_\xi\circ \sigma^{\Theta}_x,\;\text{for $x\in\hat{\mathcal{A}}$ and $\xi\in\mathcal{B}$}.
$
\item Using these projections define $\mathbb{R}$ as in Definition \ref{MainForcingRGolshani}. It is easy to check that $\mathbb{R}$ forces the statements of propositions \ref{projectionUMagidor} and \ref{PropertiesOfV[R]Magidor}. For each $x\in\hat{\mathcal{A}}$, define $\mathbb{R}\upharpoonright x$ as in Definition \ref{truncationsR}. Now assume that $\mathbb{R}$ forces a failure of $\TP(\lambda)$.  Arguing in the same fashion as in Lemma \ref{TruncationsAndAronszajnTreeMagidor} one obtains a set $x^*\in\hat{\mathcal{A}}$, $x_0\subsetneq x^*$ for which $\mathbb{R}\upharpoonright x^*$ forces the same failure.

\item Let $\pi^*$ be a bijection between $\mathbb{A}_{x^*}$ and $\mathbb{A}_\lambda$ extending $\pi$. Set $\mathcal{U}_\lambda^{\pi^*}:=\pi^*{(\dot{\mathcal{U}}_{x^*})}_{\pi^*[G\upharpoonright x^*]}$. It is evident that this is a coherent sequence of measures
which (pointwise) extends  $\mathcal{U}^\pi_{\lambda}$.
Set $\mathbb{M}^{\pi^*}_\lambda:=\mathbb{M}_{\mathcal{U}^{\pi^*}_\lambda}$. %Let us denote by $\mathcal{U}_\lambda^{\pi^*}(\alpha,\beta)$ the measures appearing in $\mathcal{U}^{\pi^*}_\lambda$.

\item Argue as in Lemma \ref{R*RtruncatedAreIsomorphicMagidor} to show that $\pi^*$ extends to an isomorphism between $\mathbb{A}_{x^*}\ast \dot{\mathbb{M}}_{x^*}$ and $\mathbb{A}_\lambda\ast\dot{\mathbb{M}}^{\pi^*}_\lambda$, and use it to define $\mathbb{R}^*$ as in
Definition \ref{DefinitionR*Magidor}. It can be argued that $\mathbb{R}^*$ and $\mathbb{R}\upharpoonright x^*$ are isomorphic, hence $\mathbb{R}^*$ forces the existence of a $\lambda$-Aronszajn tree. From this point on it can be checked that  all the arguments of Section \ref{TPkappa++SectionMagidor} concerning $\mathbb{R}^*$ and its truncations $\mathbb{R}^*\upharpoonright\xi$ still apply in the current context.
\end{enumerate}

We close this section with the following open question:

\begin{quest}
\rm{What is the exact consistency strength of the conclusion of Theorem \ref{main theorem 1}? }
\end{quest}

\section{A word about $\TP(\kappa^{+})$ in $V^\mathbb{R}$ }\label{NotTPkappa^+}
Assume that $\kappa$ is a strong cardinal. In particular, $\kappa^{<\kappa}=\kappa$ and thus $\square^*_\kappa$ holds. Since $\mathbb{A}_{\Theta}\ast\dot{\mathbb{M}}_\Theta$ preserves $\kappa^+$ it follows that $\one_{\mathbb{A}_\Theta\ast \dot{\mathbb{M}}_\Theta}\Vdash_{\mathbb{A}_\Theta\ast \dot{\mathbb{M}}_\Theta}\text{$``\square^*_\kappa$ holds''}$. Thus, by Proposition \ref{projectionUMagidor}(3), $\one_{\mathbb{R}}\Vdash_{\mathbb{R}}\text{$``\square^*_\kappa$ holds''}$. Finally, Jensen's Theorem yields $\one_{\mathbb{R}}\Vdash_{\mathbb{R}}\text{$``\TP(\kappa^+)$ fails''}$.

\medskip

\textbf{Acknowledgments:} The authors want to thank the anonymous referee for his/her carefully reading and for pointing out some mistakes in former versions of the paper.

\bibliographystyle{alpha}
\bibliography{biblio2}

\end{document}